\documentclass[12pt]{article}
\usepackage{graphicx}
\usepackage{amssymb, amsthm, amsmath, amsfonts}
\usepackage{mathrsfs,microtype}
\usepackage{epstopdf, color}
\usepackage{tikz}
\usepackage{hyperref}
\usetikzlibrary{arrows,decorations.pathmorphing,decorations.footprints,fadings,calc,trees,mindmap,shadows,decorations.text,patterns,positioning,shapes,matrix,fit,through,decorations.pathreplacing}
\usepackage[margin=1.3in]{geometry}

\theoremstyle{plain}
\newtheorem{theorem}{Theorem}[section]
\newtheorem{lemma}[theorem]{Lemma}
\newtheorem{corollary}[theorem]{Corollary}

\newtheorem{proposition}[theorem]{Proposition}
\newtheorem{question}[theorem]{Question}

\theoremstyle{definition}

\newtheorem*{convention}{Convention}

\theoremstyle{remark}
\newtheorem*{remark}{Remark}

\newcommand{\calG}{\mathcal{G}}
\newcommand{\Hom}{\text{Hom}}

\DeclareMathOperator{\Tr}{Tr}

\title{Extremal $H$-colorings of graphs with fixed minimum degree}
\author{John Engbers\thanks{john.engbers@marquette.edu; Department of Mathematics, Statistics and Computer Science, Marquette University, Milwaukee, WI 53201}}

\date{\today }%\thanks{$\{$jengbers, dgalvin1$\}$@nd.edu; Department of Mathematics,
\begin{document}

\maketitle

\begin{abstract}
For graphs $G$ and $H$, a homomorphism from $G$ to $H$, or $H$-coloring of $G$, is a map from the vertices of $G$ to the vertices of $H$ that preserves adjacency.  When $H$ is composed of an edge with one looped endvertex, an $H$-coloring of $G$ corresponds to an independent set in $G$.  Galvin showed that, for sufficiently large $n$, the complete bipartite graph $K_{\delta,n-\delta}$ is the $n$-vertex graph with minimum degree $\delta$ that has the largest number of independent sets.

In this paper, we begin the project of generalizing this result to arbitrary $H$.  %It turns out that $K_{\delta,n-\delta}$ is not always the graph that maximizes the number of $H$-colorings from among all $n$-vertex graphs with minimum degree $\delta$.  In particular, w
Writing $\hom(G,H)$ for the number of $H$-colorings of $G$, we show that for fixed $H$ and $\delta = 1$ or $\delta = 2$,
\[
\hom(G,H) \leq \max \{ \hom(K_{\delta+1},H)^{\frac{n}{\delta+1}}, \hom(K_{\delta,\delta},H)^{\frac{n}{2\delta}}, \hom(K_{\delta,n-\delta},H)\}
\]
for any $n$-vertex $G$ with minimum degree $\delta$ (for sufficiently large $n$). We also provide examples of $H$ for which the maximum is achieved by $\hom(K_{\delta+1},H)^{\frac{n}{\delta+1}}$ and other $H$ for which the maximum is achieved by $\hom(K_{\delta,\delta},H)^{\frac{n}{2\delta}}$.  
For $\delta \geq 3$ (and sufficiently large $n$), we provide a infinite family of $H$ for which $\hom(G,H) \leq \hom(K_{\delta,n-\delta},H)$ for any $n$-vertex $G$ with minimum degree $\delta$.  The results generalize to weighted $H$-colorings.

\end{abstract}

%=============================================================================
\section{Introduction and statement of results}
%=============================================================================

Let $G=(V(G),E(G))$ be a finite simple graph.  A \emph{homomorphism} from $G$ to a finite graph $H = (V(H),E(H))$ (without multi-edges but perhaps with loops) is a map from $V(G)$ to $V(H)$ that preserves edge adjacency.  We write
\[
\Hom(G,H) = \{f:V(G) \to V(H)\,\, | \,\,v \sim_G w \implies f(v) \sim_H f(w)\}
\]
for the set of all homomorphisms from $G$ to $H$, and $\hom(G,H)$ for $|\Hom(G,H)|$.  All graphs mentioned in this paper will be finite without multiple edges.  Those denoted by $G$ will always be loopless, while those denoted by $H$ may possibly have loops.  We will also assume that $H$ has no isolated vertices.

Graph homomorphisms generalize a number of important notions in graph theory.  When $H = H_{\text{ind}}$, the graph consisting of a single edge and a loop on one endvertex, elements of $\Hom(G,H_{\text{ind}})$ can be identified with the independent sets in $G$.  When $H = K_{q}$, the complete graph on $q$ vertices, elements of $\Hom(G,K_{q})$ can be identified with the proper $q$-colorings of $G$.  Motivated by this latter example, elements of $\Hom(G,H)$ are sometimes referred to as \emph{$H$-colorings of $G$}, and the vertices of $H$ are referred to as \emph{colors}.  We will utilize this terminology throughout the paper.

In statistical physics, $H$-colorings have a natural interpretation as configurations in \emph{hard-constraint spin systems}.  Here, the vertices of $G$ are thought of as sites that are occupied by particles, with the edges of $G$ representing pairs of sites that are bonded (for example by spatial proximity).  The vertices of $H$ represent the possible spins that a particle may have, and the occupation rule is that spins appearing on sites that are bonded must be adjacent in $H$.  A valid configuration of spins on $G$ is exactly an $H$-coloring of $G$.  In the language of statistical physics, independent sets are configurations in the \emph{hard-core gas model}, and proper $q$-colorings are configurations in the \emph{zero-temperature $q$-state antiferromagnetic Potts model}.  Another example comes from the \emph{Widom-Rowlinson} graph $H=H_{\text{WR}}$, the fully-looped path on three vertices.  If the endpoints of the path represent different particles and the middle vertex represents empty space, then the Widom-Rowlinson graph models the occupation of space by two mutually repelling particles.

\medskip

Fix a graph $H$.  A natural extremal question to ask is the following: for a given family of graphs $\calG$, which graphs $G$ in $\calG$ maximize $\hom(G,H)$?  If we assume that all graphs in $\calG$ have $n$ vertices, then there are several cases where this question has a trivial answer.  First, if $H = K_{q}^{\text{loop}}$, the fully looped complete graph on $q$ vertices, then every map $f:V(G) \to V(H)$ is an $H$-coloring (and so $\hom(G,K_{q}^{\text{loop}}) = q^{n}$).  Second, if the empty graph $\overline{K}_{n}$ is contained in $\calG$, then again every map $f:V(\overline{K}_n) \to V(H)$ is an $H$-coloring (and so $\hom(\overline{K}_n,H) = |V(H)|^n$).  Motivated by this second trivial case, it is interesting to consider families $\calG$ for which each $G \in \calG$ has many edges.

For the family of $n$-vertex $m$-edge graphs, this question was first posed for $H = K_{q}$ around 1986, independently, by Linial \cite{Linial} and Wilf \cite{Wilf}.  Lazebnik provided an answer for $q=2$ \cite{Lazebnik}, but for general $q$ there is still not a complete answer.  However, much progress has been made (see \cite{LohPikhurkoSudakov} and the references therein).  Recently, Cutler and Radcliffe answered this question for $H=H_{\text{ind}}$, $H=H_{\text{WR}}$, and some other small $H$ \cite{CutlerRadcliffe2,CutlerRadcliffe3}.  
A feature of the family of $n$-vertex, $m$-edge graphs emerging from the partial results mentioned is that there seems to be no uniform answer to the question, ``which $G$ in the family maximizes $\hom(G,H)$?'', with the answers depending very sensitively on the choice of $H$.

%Many of the results in this family generally require a different set of extremal graphs for each choice of $H$.

Another interesting family to consider is the family of $n$-vertex $d$-regular graphs.  Here, Kahn \cite{Kahn1} used entropy methods to show that every \emph{bipartite} graph $G$ in this family satisfies $\hom(G,H_{\text{ind}}) \leq \hom(K_{d,d},H_{\text{ind}})^{\frac{n}{2d}}$, where $K_{d,d}$ is the complete bipartite graph with $d$ vertices in each partition class.  Notice that when $2d|n$ this bound is achieved by $\frac{n}{2d} K_{d,d}$, the disjoint union of $n/2d$ copies of $K_{d,d}$.  Galvin and Tetali \cite{GalvinTetali} generalized this entropy argument, showing that for \emph{any} $H$ and any bipartite $G$ in this family, 
\begin{equation}\label{eqn-d regular bound}
\hom(G,H) \leq \hom(K_{d,d},H)^{\frac{n}{2d}}.
\end{equation}

Kahn conjectured that (\ref{eqn-d regular bound}) should hold for $H=H_{\text{ind}}$ for all (not necessarily bipartite) $G$, and Zhao \cite{Zhao} resolved this conjecture affirmatively, 
deducing the general result from the bipartite case.  Interestingly, (\ref{eqn-d regular bound}) does not hold for general $H$ when biparticity is dropped, as there are examples of $n$, $d$, and $H$ for which $\frac{n}{d+1} K_{d+1}$, the disjoint union of $n/(d+1)$ copies of the complete graph $K_{d+1}$, maximizes the number of $H$-colorings of graphs in this family.  (For example, take $H$ to be the disjoint union of two looped vertices; here $\log_2 (\hom(G,H))$ equals the number of components of $G$.)  %Galvin proposes the following conjecture in \cite{GalvinHColoringRegularGraphs}.
In \cite{GalvinHColoringRegularGraphs}, Galvin proposes a conjecture that for any $n$-vertex $d$-regular graph $G$ and any graph $H$,
\begin{equation}\label{conj-regular graphs}
\hom(G,H) \leq \max\{\hom(K_{d+1},H)^{\frac{n}{d+1}}, \hom(K_{d,d},H)^{\frac{n}{2d}}\}.
\end{equation}

%\begin{conjecture}\label{conj-regular graphs}
%Let $G$ be an $n$-vertex $d$-regular graph.  Then, for any $H$,
%\[
%\hom(G,H) \leq \max\{\hom(K_{d+1},H)^{\frac{n}{d+1}}, \hom(K_{d,d},H)^{\frac{n}{2d}}\}.
%\]
%\end{conjecture}

When $2d(d+1) | n$, this bound is achieved by either $\frac{n}{2d} K_{d,d}$ or $\frac{n}{d+1} K_{d+1}$. While (\ref{conj-regular graphs}) holds for a large class of $H$ (see e.g. \cite{Zhao,Zhao2}) and for any fixed $H$ asymptotically in $d$ (see \cite{GalvinHColoringRegularGraphs,GalvinProperColoringRegularGraphs}), the result is actually not true for all $n$, $d$, and $H$. %, the disjoint union of $n/(d+1)$ copies of $K_{d+1}$. %, the complete graph on $d+1$ vertices.  
%Evidence for this conjecture is given by Zhao \cite{Zhao,Zhao2}, who provided a large class of $H$ for which $\hom(G,H) \leq \hom(K_{d,d},H)^{\frac{n}{2d}}$.  Galvin \cite{GalvinHColoringRegularGraphs,GalvinProperColoringRegularGraphs} provides further results for various $H$ (including triples $(n,d,H)$ for which $\hom(G,H) \leq \hom(K_{d+1},H)^{\frac{n}{d+1}}$) and asymptotic evidence for the conjecture. 
In particular, Galvin and Sernau \cite{GalvinSernau} have found an $H$ and a $G$ (for each $d \geq 5$, where the graphs $H$ and $G$ depend on $d$) so that $\hom(G,H) > \max\{\hom(K_{d+1},H)^{\frac{n}{d+1}}, \hom(K_{d,d},H)^{\frac{n}{2d}}\}.$  %It would be interesting to see which $n$, $d$, and $H$ make the conjecture true.

%\begin{question}\label{ques-regular graphs}
%Let $G$ be an $n$-vertex $d$-regular graph.  For which $n$, $d$, and $H$ do we have
%\[
%\hom(G,H) \leq \max\{\hom(K_{d+1},H)^{\frac{n}{d+1}}, \hom(K_{d,d},H)^{\frac{n}{2d}}\}?
%\]
%\end{question}

It is clear, however, that  (\ref{conj-regular graphs}) holds when $d=1$, since the graph consisting of $n/2$ disjoint copies of an edge is the only $1$-regular graph on $n$ vertices.  We will prove (\ref{conj-regular graphs}) holds for $d=2$ and also characterize the cases of equality. 

%\begin{theorem}\label{thm-2 regular graphs}
%Let $G$ be an $n$-vertex $2$-regular graph.  Then, for any $H$,
%\[
%\hom(G,H) \leq \max \{\hom(C_{3},H)^{\frac{n}{3}}, \hom(C_{4},H)^{\frac{n}{4}}\}.
%\]
%If $H \neq K_{q}^{\text{loop}}$, then the only graphs achieving equality are $G = \frac{n}{3} C_3$ (when $\hom(C_3,H)^{\frac{n}{3}} > \hom(C_4,H)^{\frac{n}{4}}$), $G = \frac{n}{4} C_4$ (when $\hom(C_3,H)^{\frac{n}{3}} < \hom(C_4,H)^{\frac{n}{4}}$), or the disjoint union of copies of $C_3$ and copies of $C_4$ (when $\hom(C_3,H)^{\frac{n}{3}} = \hom(C_4,H)^{\frac{n}{4}}$).
%\end{theorem}

\begin{theorem}\label{thm-2 regular graphs}
Let $G$ be an $n$-vertex $2$-regular graph.  Then, for any $H$,
\[
\hom(G,H) \leq \max \{\hom(C_{3},H)^{\frac{n}{3}}, \hom(C_{4},H)^{\frac{n}{4}}\}.
\]
If $H \neq K_{q}^{\text{loop}}$, the only graphs achieving equality are $G = \frac{n}{3} C_3$ (when $\hom(C_3,H)^{\frac{1}{3}} > \hom(C_4,H)^{\frac{1}{4}}$), $G = \frac{n}{4} C_4$ (when $\hom(C_3,H)^{\frac{1}{3}} < \hom(C_4,H)^{\frac{1}{4}}$), or the disjoint union of copies of $C_3$ and copies of $C_4$ (when $\hom(C_3,H)^{\frac{1}{3}} = \hom(C_4,H)^{\frac{1}{4}}$).
\end{theorem}

It is possible for each of the equality conditions in Theorem \ref{thm-2 regular graphs} to occur.  The first two situations arise when $H$ is a disjoint union of two looped vertices and $H=K_2$, respectively.  For the third situation, we utilize that if $G$ is connected and $H$ is the disjoint union of $H_1$ and $H_2$, then $\hom(G,H) = \hom(G,H_1) + \hom(G,H_2)$.  Letting $H$ be the disjoint union of $8$ copies of a single looped vertex and and 4 copies of $K_2$ gives $\hom(C_3,H)^{\frac{1}{3}} = \hom(C_4,H)^{\frac{1}{4}} = 2$.

\medskip

Another natural and related family to study is $\calG(n,\delta)$, the set of all $n$-vertex graphs with minimum degree $\delta$.  Our question here becomes: for a given $H$, which $G \in \calG(n,\delta)$ maximizes $\hom(G,H)$?  Since removing edges increases the number of $H$-colorings, it is tempting to believe that the answer to this question will be a graph that is $\delta$-regular (or close to $\delta$-regular).  This in fact is not the case, even for $H = H_{\text{ind}}$.  The following result appears in \cite{GalvinTwoProblems}.

\begin{theorem}\label{thm-Galvin Ind Sets}
For $\delta \geq 1$, $n \geq 8\delta^2$, and $G \in \calG(n,\delta)$, we have
\[
\hom(G,H_{{\rm ind}}) \leq \hom(K_{\delta,n-\delta},H_{{\rm ind}}),
\]
with equality only for $G = K_{\delta,n-\delta}$.
\end{theorem}

Recently, Cutler and Radcliffe \cite{CutlerRadcliffe4} have extended Theorem \ref{thm-Galvin Ind Sets} to the range $n \geq 2\delta$. 
Further results related to maximizing the number of independent sets of a fixed size for $G \in \calG(n,\delta)$ can be found in e.g. \cite{AlexanderCutlerMink,AlexanderMink,EngbersGalvin3,LawMcDiarmid}, with a complete answer to this question for $n \geq 2\delta$ given by Gan, Loh, and Sudakov \cite{GanLohSudakov}.

With the results on regular graphs (both for a large class of $H$ and also for any fixed $H$ asymptotically) and Theorem \ref{thm-Galvin Ind Sets} in mind, the following question is natural.

\begin{question}\label{conj-H Colorings Min Degree}
%Fix $\delta \geq 1$ and $H$.  
For which fixed $\delta \geq 1$ and $H$ does there exists a constant $c(\delta,H)$ (depending on $\delta$ and $H$) such that for $n \geq c(\delta,H)$ and $G \in \calG(n,\delta)$,
\[
\hom(G,H) \leq \max \{\hom(K_{\delta+1},H)^{\frac{n}{\delta+1}},  \hom(K_{\delta,\delta},H)^{\frac{n}{2\delta}}, \hom(K_{\delta,n-\delta},H) \}?
\]
\end{question}

%This question stands in marked contrast to the situation for the family of $n$-vertex $m$-edge graphs, where each choice of $H$ seems to create a different set of extremal graphs.  Here, we conjecture that for \emph{any} $H$, one of exactly three situations can occur.  
Notice for $2\delta(\delta+1) |n$ and $n$ large, we have that for the three examples $H=E_2^{\text{loop}}$ (the disjoint union of two looped vertices), $H=K_2$, and $H=H_{\text{ind}}$, 
%since for $H$ consisting of a disjoint union of two looped vertices, $H=K_2$, and $H=H_{\text{ind}}$, 
 the number of $H$-colorings of a graph $G \in \calG(n,\delta)$ is maximized by $G = \frac{n}{\delta+1} K_{\delta+1}$, $G=\frac{n}{2\delta} K_{\delta,\delta}$, and $G=K_{\delta,n-\delta}$, respectively.

The purpose of this paper is to make progress toward answering Question \ref{conj-H Colorings Min Degree}.  We provide an answer to the question for $\delta=1$ and $\delta=2$, and characterize the graphs that achieve equality.  We also find an infinite family of $H$ for which $\hom(G,H) \leq \hom(K_{\delta,n-\delta},H)$ for all $G \in \calG(n,\delta)$ (for sufficiently large $n$), with equality only for $G = K_{\delta,n-\delta}$.  Before we formally state these theorems, we highlight the degree conventions and notations that we will follow for the remainder of the paper.

\begin{convention}
For $v \in V(H)$, let $d(v)$ denote the degree of $v$, where loops count \emph{once} toward the degree.  While $\delta$ will always refer to the minimum degree of a graph $G$, $\Delta$ will always denote the maximum degree of a graph $H$ (unless explicity stated otherwise). 
\end{convention}

Notice that with this convention we have $\hom(K_{1,a},H) = \sum_{v \in V(H)} d(v)^{a}$; in particular, we have $\hom(K_2,H) = \sum_{v \in V(H)} d(v)$. In the context of these simple expressions for $\hom(K_{1,a},H)$, the degree convention for loops is natural in this setting.

\begin{theorem}\label{thm-Strong delta1}
{\em \textbf{($\delta=1$).}} Fix $H$, $n\geq 2$ and $G \in \calG(n,1)$. 
		\begin{enumerate}
			\item Suppose that $H \neq K_{\Delta}^{\text{loop}}$ satisfies $\hom(K_2,H) \geq \Delta^2$.  Then 
			\[
			\hom(G,H) \leq \hom(K_2,H)^{\frac{n}{2}},
			\]
			with equality only for $G = \frac{n}{2} K_{2}$.
			
			\item Suppose that $H$ satisfies $\hom(K_2,H) < \Delta^2$, and let $n_0 = n_0(H)$ be the smallest integer in $\{3,4,\ldots\}$ satisfying $\hom(K_2,H) < \hom(K_{1,n_0-1},H)^{\frac{2}{n_0}}$. %\left(\sum_{v \in V(H)} d(v)^{n_0-1}\right)^{\frac{2}{n_0}}$.
			
			\begin{enumerate}
				\item If $2 \leq n < n_0$, then 
				\[
				\hom(G,H) \leq \hom(K_{2},H)^{\frac{n}{2}},
				\]
				with equality only for $G=\frac{n}{2} K_2$ [unless $n=n_0-1$ and $\hom(K_2,H)  = \hom(K_{1,n_0-2},H)^{\frac{2}{n_0-1}}$, %\left(\sum_{v \in V(H)} d(v)^{n_0-2}\right)^{\frac{2}{n_0-1}}$, 
				in which case $G = K_{1,n-1}$ also achieves equality].
				\item If $n \geq n_0$, then
				\[
				\hom(G,H) \leq \hom(K_{1,n-1},H),
				\]
				with equality only for $G = K_{1,n-1}$.
			\end{enumerate}
		\end{enumerate}
\end{theorem}

\begin{remark}
Since $\hom(K_2,H) = \sum_{v \in V(H)} d(v)$, the conditions on $H$ in Theorem \ref{thm-Strong delta1} may also be written as $\sum_{v \in V(H)} d(v) \geq \Delta^2$ and $\sum_{v \in V(H)} d(v) < \Delta^2$. 
\end{remark}

\begin{theorem}\label{thm-Strong delta2}
{\em \textbf{($\delta=2$).}} Fix $H$. % and $G \in \calG(n,2)$.
\begin{enumerate}
	\item Suppose that $H \neq K_{\Delta}^{\text{loop}}$ satisfies $\max \{ \hom(C_3,H)^{\frac{1}{3}}, \hom(C_4,H)^{\frac{1}{4}} \} \geq \Delta$.  Then for all $n \geq 3$ and $G \in \calG(n,2)$, 
	\[
	\hom(G,H) \leq \max\{ \hom(C_3,H)^{\frac{n}{3}}, \hom(C_4,H)^{\frac{n}{4}}\},
	\]
	with equality only for $G = \frac{n}{3} C_3$ (when $\hom(C_3,H)^{\frac{1}{3}} > \hom(C_4,H)^{\frac{1}{4}}$), $G = \frac{n}{4} C_4$ (when $\hom(C_3,H)^{\frac{1}{3}} < \hom(C_4,H)^{\frac{1}{4}}$), or the disjoint union of copies of $C_3$ and copies of $C_4$ (when $\hom(C_3,H)^{\frac{1}{3}} = \hom(C_4,H)^{\frac{1}{4}}$).
	
	\item Suppose that $H$ satisfies $\max \{ \hom(C_3,H)^{\frac{1}{3}}, \hom(C_4,H)^{\frac{1}{4}} \} < \Delta$.  Then there exists a constant $c_H$ such that for $n>c_H$ and $G \in \calG(n,2)$,
	\[
	\hom(G,H) \leq \hom(K_{2,n-2},H),
	\]
	with equality only for $G = K_{2,n-2}$.

\end{enumerate}
\end{theorem}

Theorems \ref{thm-Strong delta1} and \ref{thm-Strong delta2} are easily seen to answer Question \ref{conj-H Colorings Min Degree} when $\delta=1$ and $\delta=2$, respectively.  Notice that if $G'$ is obtained from $G$ by deleting some edges from $G$, then $\hom(G,H) \leq \hom(G',H)$.  Because of this, their proofs focus on $G$ which are \emph{edge-min-critical} for $\delta$, meaning that the minimum degree of $G$ is $\delta$ and for any edge $e$ the minimum degree of $G - e$ is $\delta-1$.

The edge-min-critical graphs in $\calG(n,1)$ are disjoint unions of stars, and the proof of Theorem \ref{thm-Strong delta1} critically uses this fact.  Theorem \ref{thm-Strong delta2} relies on a structural characterization of edge-min-critical graphs in $\calG(n,2)$ (see Lemma \ref{lem-delta2structure}) and also uses Theorem \ref{thm-2 regular graphs}.  The global structure of edge-min-critical graphs in $\calG(n,\delta)$ for $\delta \geq 3$ is not very well understood.

%Define S(\delta,H) here to get a lower bound on \hom(K_{\delta,n-\delta},H).  Show how this is quite large

\medskip 

We also make some progress in the general $\delta$ case of Question \ref{conj-H Colorings Min Degree} by providing an infinite family of $H$ for which $K_{\delta,n-\delta}$ has the largest number of $H$-colorings.  %We state the result formally here.

\begin{theorem}\label{thm-ExtremalHColorings}
Fix $\delta$ and $H$. Suppose that $H$ satisfies 
%\begin{equation} \label{eqn-deltacondition}
 $\hom(K_2,H) = \sum_{v \in V(H)} d(v) < \Delta^2.$
%\end{equation}
  Then there exists a constant $c_H$ such that for all $n \geq (c_H)^\delta$ and $G \in \calG(n,\delta)$,
		\[
		\hom(G,H) \leq \hom(K_{\delta,n-\delta},H),
		\]
		with equality only for $G=K_{\delta,n-\delta}$.	 
		
		If $H$ also has the property that all vertices of degree $\Delta$ share the same $\Delta$ neighbors, then the same result holds for all $n \geq c_H \delta^2$. 

\end{theorem}

\noindent The following corollary warrants special attention, and is immediate.

\begin{corollary}\label{cor-special case}
Suppose that $H \neq K_{\Delta}^{\text{loop}}$ has a looped dominating vertex, or that $H$ satisfies $\hom(K_2,H) =  \sum_{v \in V(H)} d(v)  < \Delta^2$ and $H$ has a unique vertex of degree $\Delta$.  Then there exists a constant $c_H$ such that for all $n \geq c_H \delta^2$ and $G \in \calG(n,\delta)$,
\[
\hom(G,H) \leq \hom(K_{\delta,n-\delta},H),
\]
with equality only for $G=K_{\delta,n-\delta}$.
\end{corollary}

The graphs $H=H_{\text{ind}}$ and $H=H_{\text{WR}}$ satisfy the conditions of Corollary \ref{cor-special case}, so in particular we provide an alternate proof of Galvin's result for $H=H_{\text{ind}}$ \cite{GalvinTwoProblems}.  Another graph $H$ which satisfies the conditions of Corollary \ref{cor-special case} is the $k$-state hard-core constraint graph $H(k)$ ($k \geq 1)$, the graph with vertex set $\{0,1,\ldots,k\}$ and edge $i \sim_{H(k)} j$ if $i + j \leq k$.  This graph naturally occurs in the study of multicast communications networks, and has been considered in e.g. \cite{GalvinMartinelliRamananTetali,MitraRamananSengupta}.  

Notice that the condition on $H$ in Theorem \ref{thm-ExtremalHColorings} is necessary but not sufficient for $\hom(G,H) \leq \hom(K_{\delta,n-\delta},H)$ for all $G \in \calG(n,\delta)$.  Indeed, if $H$ is a path on three vertices with a loop on one endpoint of the path, then $\sum_{v \in V(H)} d(v) = 5$ while $\Delta = 2$.  However, for large enough $n$ and $G \in \calG(n,2)$, $\hom(G,H) \leq \hom(K_{2,n-2},H)$, as can be seen by computing $\hom(C_3,H)$, $\hom(C_4,H)$, and applying Theorem \ref{thm-Strong delta2}.

\medskip

It is also interesting to consider a maximum degree condition in addition to a minimal degree condition (see e.g. \cite{AlexanderCutlerMink,GalvinSmallDegree,Kahn1}).  Let $\calG(n,\delta,D)$ denote the set of graphs on $n$ vertices with minimum degree $\delta$ and maximum degree at most $D$.  Which graphs $G \in \calG(n,\delta,D)$ maximize $\hom(G,H)$?  For a fixed $\delta$, this question is interesting for the $H$ with the property that $\hom(G,H) \leq \hom(K_{\delta,n-\delta},H)$ for all $G \in \calG(n,\delta)$, as $K_{\delta,n-\delta} \in \calG(n,\delta,D)$ only when $D \geq n-\delta$. 

For $\delta=1$ and any $D \geq 1$, we provide an answer.

\begin{theorem}\label{thm-max degree condition}
Fix $H$ and $D \geq 1$.  For any $G \in \calG(n,1,D)$,
\[
\hom(G,H) \leq \max \{ \hom(K_{2},H)^{\frac{n}{2}},  \hom(K_{1,D},H)^{\frac{n}{1+D}} \},
\]
with the cases of equality as in Theorem \ref{thm-Strong delta1} (where $K_{1,n-1}$ is replaced by $\frac{n}{1+D}K_{1,D}$). 
\end{theorem}
%The proof of Theorem \ref{thm-max degree condition} is given in Section \ref{sec-delta1}, and again utilizes the fact that edge-critical graphs for $\delta=1$ are disjoint unions of stars.

From a statistical physics standpoint, there is a very natural family of probability distributions that can be put on $\Hom(G,H)$.  Fix a set of positive weights $\Lambda = \{\lambda_i : i \in V(H)\}$ indexed by the vertices of $H$.  We think of $\lambda_i$ as representing the likelihood of particle $i$ appearing at a site in $G$, and formalize this by giving an element $f \in \Hom(G,H)$ weight $w_\Lambda(f) = \prod_{v \in V(G)} \lambda_{f(v)}$ and probability
\[
p_{\Lambda}(f) = \frac{w_\Lambda(f)}{Z_\Lambda(G,H)},
\]
where $Z_\Lambda(G,H) = \sum_{f} w_\Lambda(f)$ is the appropriate normalizing constant (or \emph{partition function}) of the model.  %The partition function is an important topic of study for $H$-colorings.  
By taking $\lambda_i = 1$ for each $i$, $p_\Lambda(f)$ is the uniform distribution on $\Hom(G,H)$ and in this case $Z_\Lambda(G,H) = \hom(G,H)$.  

Interestingly, several proofs of structural results about $H$-colorings \emph{require} passing to the weighted model first; see e.g. \cite{EngbersGalvin1,Kahn1,Kahn3}.  Our results generalize naturally to weighted $H$-colorings. %by considering $d_\Lambda(v) = \sum_{w \sim_G v} \lambda_{w}$ (where again loops on $v$ imply that there is one $\lambda_v$ in the sum), $\Delta_\Lambda = \max_{v \in V(H)} d_\Lambda(v)$, and $\hom(G,H) = Z_\Lambda(G,H)$.  
Although the proofs of the weighted versions come with almost no extra effort, for the clarity of presentation we defer this discussion until Section \ref{sec-concludingremarks}.

The paper is laid out as follows.  In Section \ref{sec-deltageneral}, we prove Theorem \ref{thm-ExtremalHColorings} by partitioning $\calG(n,\delta)$ based on the size of a maximal matching.  Section \ref{sec-delta1} utilizes the structure of edge-min-critical graphs for $\delta=1$ to prove Theorems \ref{thm-Strong delta1} and \ref{thm-max degree condition}.   By analyzing $H$-colorings of cycles, we prove Theorem \ref{thm-2 regular graphs} in Section \ref{sec-2regular}.  Following some preliminary lemmas about edge-min-critical graphs for $\delta=2$ and $H$-colorings of paths, we prove  Theorem \ref{thm-Strong delta2} in Section \ref{sec-delta2}.  Finally, in Section \ref{sec-concludingremarks} we comment on the generalization of our results to weighted $H$-colorings and also present some related questions.

%===========================================================================
%Section general \delta
%===========================================================================

\section{Proof of Theorem \ref{thm-ExtremalHColorings}}\label{sec-deltageneral}

Suppose that we have a graph $H$ satisfying 
\begin{equation}\label{eqn-Hcondition delta}
\sum_{v \in V(H)} d(v) < \Delta^2,
\end{equation}
and let $G$ be a graph with minimum degree $\delta$.  Let $M$ be the edge set of a matching of maximum size in $G$ and $I$ the set of unmatched vertices.  

We first derive some structural properties of our graph $G$ based on $M$.  Since $M$ is maximal, $I$ forms an independent set.  Furthermore, suppose that $x_1$ and $x_2$ in $V(G)$ are matched in $M$.  If $x_1$ has at least two edges into $I$, then $x_2$ cannot be adjacent to any vertex in $I$, as this would create an augmenting path %(a path which starts and ends at distinct unmatched vertices and alternates edges in and out of $M$) 
of length 3 and therefore a matching of larger size.  In summary:
\begin{equation} \label{eqn-matching}
\begin{array}{c}
\text{At most one vertex in an edge of the matching $M$ can have degree at }\\
\text{ least }2 \text{ into }I\text{, and if one has degree at least }2 \text{ into }I \text{ then the other}\\ 
\text{ has degree }0\text{ into }I.  
\end{array}
\end{equation}
%We partition the endpoints of the edges in $M$ into sets $J$ and $K$ according to the following rule.
  
For each edge in $M$, put the endpoint with the largest degree into $I$ in a set $J \subset V(G)$, and put the other endpoint in a set $K \subset V(G)$. (If the degrees are equal, make an arbitrary choice.)  A schematic picture of $G$ is shown in Figure \ref{fig-matching}; there are at most $|K| = |M|$ total edges between $I$ and $K$.

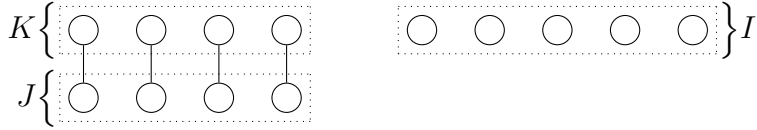
\begin{figure}[ht]	
\begin{center}

	\begin{tikzpicture}[scale=.9]
	\node (v1) at (1,1) [circle,draw,label=180:$J \Big\{$] {};
	\node (w1) at (1,2) [circle,draw,label=180:$K \Big\{$] {};
	\node (v2) at (2,1) [circle,draw] {};
	\node (w2) at (2,2) [circle,draw] {};
	\node (v3) at (3,1) [circle,draw] {};
	\node (w3) at (3,2) [circle,draw] {};
	\node (v4) at (4,1) [circle,draw] {};
	\node (w4) at (4,2) [circle,draw] {};
	\draw (v1) -- (w1);
	\draw (v2) -- (w2);
	\draw (v3) -- (w3);
	\draw (v4) -- (w4);
	\node (i1) at (6,2) [circle,draw] {};
	\node (i2) at (7,2) [circle,draw] {};
	\node (i3) at (8,2) [circle,draw] {};
	\node (i4) at (9,2) [circle,draw] {};
	\node (i5) at (10,2) [circle,draw,label=0:$\Big\} I$] {};
	
	\draw[dotted] (.65,1.65) -- (.65,2.35);
	\draw[dotted] (.65,2.35) -- (4.35,2.35);
	\draw[dotted] (4.35,2.35) -- (4.35,1.65);
	\draw[dotted] (4.35,1.65) -- (.65,1.65);
	\draw[dotted] (.65,.65) -- (.65,1.35);
	\draw[dotted] (.65,1.35) -- (4.35,1.35);
	\draw[dotted] (4.35,1.35) -- (4.35,.65);
	\draw[dotted] (4.35,.65) -- (.65,.65);
	\draw[dotted] (5.65,1.65) -- (5.65,2.35);
	\draw[dotted] (5.65,2.35) -- (10.35,2.35);
	\draw[dotted] (10.35,2.35) -- (10.35,1.65);
	\draw[dotted] (10.35,1.65) -- (5.65,1.65);
	
	\end{tikzpicture}

	\caption{The relevant structure for $G$.}
	\label{fig-matching}
\end{center}

\end{figure}

Also, if there are more than $|M|$ vertices in $I$ that are adjacent to both endpoints of some edge in $M$, then by the pigeonhole principle there are distinct $y_1, y_2 \in I$ that are adjacent to both endpoints of some fixed edge in $M$.  This would force both endpoints of an edge in $M$ to have degree at least 2 into $I$, contradicting (\ref{eqn-matching}). Therefore we have:
\begin{equation}\label{eqn-matching2}
\begin{array}{c} 
\text{There are at most }|M| \text{ vertices in }I\text{ adjacent to both endpoints of}\\
\text{ some edge in }M.
\end{array}
\end{equation}
 In particular, suppose $n \geq 3\delta-2$.  Then if $|M| < \delta$ we have $|I| \geq \delta$.  Since each $x \in I$ has at least $\delta$ neighbors to $M$, each $x \in I$ is adjacent to both endpoints of some edge in $M$.  Since $|I| \geq \delta > |M|$, this contradicts (\ref{eqn-matching2}).  Therefore if $n \geq 3\delta-2$ we have $|M| \geq \delta$ for all $G \in \calG(n,\delta)$.  We will first analyze the graphs where $|M|=\delta$ and then the graphs where $|M| > \delta$.

\medskip

 \textbf{Case 1:} Suppose that $|M| = \delta$ and $n > 3\delta$, so by  (\ref{eqn-matching2}) at most $\delta$ vertices in $I$ are adjacent to both endpoints of some edge in $M$.  Then there is at least one vertex in $I$ that is adjacent to exactly one endpoint of each edge in $M$.  However, this shows that \emph{no} vertex in $I$ can be adjacent to both endpoints of any edge in $M$, since any vertex in $I$ adjacent to both vertices of an edge in $M$ would force one endpoint of $M$ to have degree at least $2$ into $I$ and the other endpoint of $M$ to have degree at least $1$ into $I$ (contradicting (\ref{eqn-matching})).  It follows that each vertex in $I$ must be adjacent to each vertex in $J$, and so by (\ref{eqn-matching}) there are no edges between $I$ and $K$.  
 
 Now suppose $k_1, k_2 \in K$ with $k_1 \sim k_2$.  Then there exist distinct $j_1, j_2 \in J$ with $k_1 \sim_M j_1$ and $k_2 \sim_M j_2$.  Letting $i_1$ and $i_2$ denote any two distinct vertices in $I$ (and recalling that everything in $I$ is adjacent to everything in $J$), $i_1 \sim j_1 \sim k_1 \sim k_2 \sim j_2 \sim i_2$ is an augmenting path of length $5$, which contradicts the maximality of $M$.  Therefore $K$ is an independent set and so $K \cup I$ is an independent set.  Since $G$ has minimum degree $\delta$, every vertex in $K \cup I$ is adjacent to every vertex in $J$, and so $G$ must be the complete bipartite graph $K_{\delta,n-\delta}$ with some edges added to the size $\delta$ partition class.
 
%Additionally, $K$ is an independent set, since any edge in $K$ extends to an augmenting path of length 5 via the neighbors in $M$ of the endpoints of that edge and any two distinct vertices in $I$ (recall that each vertex in $I$ is adjacent to every vertex in $J$).  Therefore, $G$ must be the complete bipartite graph $K_{\delta, n-\delta}$ with some edges added in the size $\delta$ partition class.  %So for $n > 3\delta$, $K_{\delta,n-\delta}$ is the only edge-critical graph satisfying $|M| \leq \delta$.  

We now show that adding any edge to the size $\delta$ partition class in $K_{\delta,n-\delta}$ will strictly decrease the number of $H$-colorings.  Since $H$ cannot contain $K_{\Delta}^{\text{loop}}$ (by (\ref{eqn-Hcondition delta})), there are two non-adjacent neighbors of a vertex in $H$ with degree $\Delta$. (It is possible that the two non-adjacent neighbors here are actually the same vertex in $H$, where here the non-adjacency means that no loop is present.)  If any edge is added to the size $\delta$ partition class in $K_{\delta,n-\delta}$, then it is impossible for any $H$-coloring to color the endpoints of that edge with the non-adjacent vertices in $H$, but such a coloring is possible in $K_{\delta,n-\delta}$.  Since any $H$-coloring of $K_{\delta,n-\delta}$ with an edge added is an $H$-coloring of $K_{\delta,n-\delta}$, this shows that the number of $H$-colorings strictly decreases whenever an edge is added to $K_{\delta,n-\delta}$. %A similar argument shows that no edge can be added to the size $n-\delta$ partition class.  

In summary, 
we have shown that if $G$ satisfies $n > 3\delta$, $|M| \leq \delta$, and $G \neq K_{\delta,n-\delta}$ then $\hom(G,H) < \hom(K_{\delta,n-\delta},H)$.

\medskip

 \textbf{Case 2:} Now suppose that $|M| = k \geq \delta + 1$. We will show that for large enough $n$ we have $\hom(G,H) < \hom(K_{\delta,n-\delta},H)$, which will complete the proof.

Let $S(\delta, H)$ denote the vectors in $V(H)^\delta$ with the property that the elements of the vector have $\Delta$ common neighbors, and let $s(\delta,H) = |S(\delta,H)|$.  (Note that $S(\delta,H) \neq \emptyset$, since if $v \in V(H)$ with $d(v) = \Delta$ then $(v,v,\ldots,v) \in S(\delta,H)$.)  We obtain a lower bound on $\hom(K_{\delta,n-\delta},H)$ by coloring the size $\delta$ partition class using an element of $S(\delta,H)$, and then independently coloring the vertices in the size $n-\delta$ partition class using the $\Delta$ common neighbors. This gives
\[
\hom(K_{\delta,n-\delta},H) \geq s(\delta,H) \Delta^{n-\delta}.
\]
We will show that for $n$ large and $k \geq \delta+1$,  $\hom(G,H) < s(\delta,H) \Delta^{n-\delta}$.  %Since the number of $H$-colorings is uniquely maximized by $K_{\delta,n-\delta}$ for $k \leq \delta$, we obtain the result.

Our initial coloring scheme will be to color $J$ arbitrarily first, then $K$, then $I$, keeping track of an upper bound on the number of choices we have for the color at each vertex.  If a vertex in $J$ is colored with $v \in V(H)$, its neighbor in $M$ has at most $d(v)$ choices for a color.  Since each vertex in $I$ is adjacent to some vertex in $J \cup K$, there are at most $\Delta$ choices for the color of each vertex in $I$.  This gives
\begin{equation*} %\label{eqn-interm bound}
\hom(G,H) \leq \left( \sum_{v \in V(H)} d(v) \right)^k \Delta^{n-2k} = \Delta^{n} \left( \frac{\sum_{v \in V(H)} d(v)}{\Delta^2} \right)^{k}.
\end{equation*}

Recalling that $H$ satisfies (\ref{eqn-Hcondition delta}), if $k > \delta \log \Delta / \log \left(\frac{\Delta^2}{\sum_{v \in V(H)} d(v)} \right) = C_{H} \delta$ this upper bound is smaller than $\Delta^{n-\delta}$.  So we may further assume that $\delta + 1 \leq k \leq C_{H} \delta$.

\medskip

Let $I' \subset I$ be the set of vertices in $I$ with neighbors exclusively in $J$, so by (\ref{eqn-matching}) we have $|I'| \geq n-3k$.  Since each vertex in $I'$ has at least $\delta$ neighbors in $J$, we imagine each $x \in I'$ picking a subset of size $\delta$ from $J$ (from among the $\binom{k}{\delta}$ possibilities).  By the pigeonhole principle there is a set $J_1 \subset J$ with $|J_1| = \delta$ and at least $(n-3k)/\binom{k}{\delta}$ vertices in $I'$ adjacent to each vertex in $J_1$.  See Figure \ref{fig-large complete graph}.

\begin{figure}[ht]
\begin{center}

	\begin{tikzpicture}[scale=.9]
	\node (v1) at (1,1) [circle,draw,label=180:$J \Big\{$] {};
	\node (w1) at (1,2) [circle,draw,label=180:$K \Big\{$] {};
	\node (v2) at (2,1) [circle,draw] {};
	\node (w2) at (2,2) [circle,draw] {};
	\node (v3) at (3,1) [circle,draw] {};
	\node (w3) at (3,2) [circle,draw] {};
	\node (v4) at (4,1) [circle,draw] {};
	\node (w4) at (4,2) [circle,draw] {};
	\draw (v1) -- (w1);
	\draw (v2) -- (w2);
	\draw (v3) -- (w3);
	\draw (v4) -- (w4);
	\node (i1) at (6,2) [circle,draw] {};
	\node (i2) at (7,2) [circle,draw] {};
	\node (i3) at (8,2) [circle,draw] {};
	\node (i4) at (9,2) [circle,draw] {};
	\node (i5) at (10,2) [circle,draw,label=0:$\Big\} I$] {};
	
	\foreach \from/\to in {v4/i1,v4/i2,v4/i3,v3/i1,v3/i2,v3/i3}
	\draw (\from) -- (\to);
	
\draw [thick,decoration={brace,mirror,raise=10pt},decorate] (2.75,1) --  node[below=10pt]{$J_1$}(4.25,1);	
	
\draw [thick,decoration={brace,raise=10pt},decorate] (5.75,2) --  node[above=10pt]{$I'$}(8.25,2);

	\draw[dotted] (.65,1.65) -- (.65,2.35);
	\draw[dotted] (.65,2.35) -- (4.35,2.35);
	\draw[dotted] (4.35,2.35) -- (4.35,1.65);
	\draw[dotted] (4.35,1.65) -- (.65,1.65);
	\draw[dotted] (.65,.65) -- (.65,1.35);
	\draw[dotted] (.65,1.35) -- (4.35,1.35);
	\draw[dotted] (4.35,1.35) -- (4.35,.65);
	\draw[dotted] (4.35,.65) -- (.65,.65);
	\draw[dotted] (5.65,1.65) -- (5.65,2.35);
	\draw[dotted] (5.65,2.35) -- (10.35,2.35);
	\draw[dotted] (10.35,2.35) -- (10.35,1.65);
	\draw[dotted] (10.35,1.65) -- (5.65,1.65);
	
	\end{tikzpicture}
	\label{fig-large complete graph}
\end{center}
\caption{Vertices in $I'$ adjacent to every vertex in $J_1$.}
\end{figure}
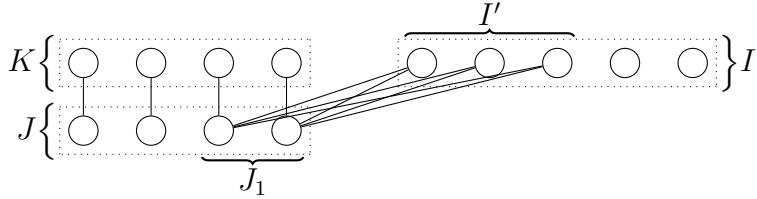

We partition the $H$-colorings of $G$ based on whether the colors on $J_1$ form a vector in $S(\delta,H)$ or not.  If they do, then we next color $J\setminus J_1$, then $K$, and then $I$, giving at most
\[
s(\delta,H) \cdot \left( \sum_{v \in V(H)} d(v) \right)^{k-\delta} \cdot \Delta^\delta \cdot \Delta^{n-2k}
\]
$H$-colorings of $G$ of this type.

If the colors on $J_1$ do not form a vector in $S(\delta,H)$, then we have at least $(n-3k)/\binom{k}{\delta}$ vertices in $I$ (namely those in $I'$) that have at most $\Delta-1$ choices for their color (here we're using that all edges are present between $I'$ and $J_1$).  Utilizing only this restriction, coloring $J\setminus J_1$, then $K$, then $I$ gives at most
\[
\left( \sum_{v \in V(H)} d(v) \right)^{k} \Delta^{n-2k} \left( \frac{\Delta-1}{\Delta} \right)^{\frac{n-3k}{\binom{k}{\delta}}}
\]
$H$-colorings of $G$ of this type.  Therefore, using $k \leq C_H \delta$ and $\binom{a}{b} \leq \left( \frac{ea}{b} \right)^b$, we have
\begin{eqnarray*}
\hom(G,H) &\leq& s(\delta,H) \cdot \left( \sum_{v \in V(H)} d(v) \right)^{k-\delta} \Delta^{n-2k+\delta} \\
&& \qquad + \left( \sum_{v \in V(H)} d(v) \right)^{k} \Delta^{n-2k} \left( \frac{\Delta-1}{\Delta} \right)^{\frac{n-3k}{\binom{k}{\delta}}} \\
&\leq& s(\delta,H) \cdot \left( \frac{\sum_{v \in V(H)} d(v)}{\Delta^2} \right)^{k-\delta} \Delta^{n-\delta} \\
&& \qquad + \left( \frac{\sum_{v \in V(H)} d(v)}{\Delta^2} \right)^{k} \Delta^{n} \left( \frac{\Delta-1}{\Delta} \right)^{\frac{n-3 C_{H} \delta}{(e \cdot C_H)^{\delta}}}
\end{eqnarray*}
so that
\begin{equation*}
\hom(G,H) \leq s(\delta,H)\Delta^{n-\delta}  \left( \frac{\sum_{v \in V(H)} d(v)}{\Delta^2} \right)^{k-\delta} \left( 1 +  r_1(\delta,H) \right)
\end{equation*}
where
\[
r_1(\delta,H) = \frac{1}{s(\delta,H)} \left( \frac{\sum_{v \in V(H)} d(v)}{\Delta} \right)^{\delta} \left( \frac{\Delta-1}{\Delta} \right)^{\frac{n-3 C_{H} \delta}{(e \cdot C_H)^{\delta}}}.
\]
For $\delta+1 \leq k \leq C_H \delta$ and $n \geq (c_H)^{\delta}$, this is smaller than $s(\delta,H) \Delta^{n-\delta}$.

\medskip

We sharpen the bounds on $n$ when all of the vertices of $H$ with degree $\Delta$ have identical neighborhoods.  (Notice that this only requires a new argument for the range $\delta+1 \leq k \leq C_H\delta$.)  Let $V_{=\Delta}(H)$ denote the set of degree $\Delta$ vertices in $H$, and so by assumption each vertex in $V_{=\Delta}(H)$ has the same $\Delta$ neighbors (and also $s(\delta,H) = |V_{=\Delta}(H)|^{\delta}$).  Our new strategy is to find a set of $\delta$ vertices in $J$ with large degree to $I$ individually instead of finding those with a large common neighborhood in $I$.  %This will be enough to limit the contribution of the $H$-colorings which color too many vertices from $J$ with a color from $V(H) \setminus V_\Delta(H)$.  

Let $J = \{x_1, \ldots, x_k\}$ and let $a_t$ denote the number of edges from $x_t$ to $I$ for each $t$.  Without loss of generality, assume $a_1 \geq a_2 \geq \cdots \geq a_k$.  Since $I$ is an independent set of size $n-2k$, there are at least $\delta(n-2k)$ edges from $I$ to $J \cup K$.  Since the degree to $I$ of each vertex in $K$ is at most $1$ and each $a_t$ is at most $n-2k$, we have
\[
\delta(n-2k)-k \leq \sum_{t=1}^{k} a_t \leq (k-\delta+1)a_{\delta} + (\delta-1)(n-2k),
\]
as $(n-2k)-k$ is a lower bound on the number of edges from $I$ to $J$, $a_\delta + a_{\delta+1} + \cdots + a_k \leq (k-\delta+1) a_{\delta}$ (by the ordering of the $a_i$'s), and $a_{\delta-1} \leq \cdots \leq a_1 \leq n-2k$.
This gives
\[
a_{\delta} \geq \frac{n-3k}{k-\delta + 1}.
\]
Now set $J_2 = \{x_1, \ldots, x_{\delta}\}$.  We first upper bound the number of $H$-colorings of $G$ that color each vertex in $J_2$ with a color from $V_{=\Delta}(H)$.  By coloring $J \setminus J_2$ arbitrarily, then coloring $K$, then coloring $I$, we have at most
\begin{equation}\label{eqn-noJ'}
s(\delta,H) \left(  \sum_{v \in V(H)} d(v) \right)^{k-\delta} \Delta^{n-2k+\delta} = s(\delta,H) \Delta^{n-\delta} \left( \frac{\sum_{v \in V(H)} d(v)}{\Delta^{2}} \right)^{k-\delta}
\end{equation}
$H$-colorings of $G$ of this type.

By similar means, we can put an upper bound on the number of $H$-colorings of $G$ that have some vertex of $J_2$ colored from $V(H) \setminus V_{=\Delta}(H)$.  Here, at least $\frac{n-3k}{k-\delta+1}$ vertices in $I$ will have at most $\Delta-1$ choices of a color for each coloring of $J \cup K$.  Using $k \leq C_H \delta$, we have at most 
\begin{eqnarray}\label{eqn-J'}
\nonumber &&\left( \sum_{v \in V(H)} d(v) \right)^{k} \Delta^{n-2k} \left( \frac{\Delta-1}{\Delta} \right)^{\frac{n-3k}{k-\delta+1}}  \\
 && \qquad \qquad \leq \Delta^{n-\delta} \left( \frac{\sum_{v \in V(H)} d(v)}{\Delta^2} \right)^{k} \Delta^{\delta} \left( \frac{\Delta-1}{\Delta} \right)^{\frac{n-3C_{H}\delta}{C_{H}\delta-\delta+1}}
\end{eqnarray}
$H$-colorings of $G$ of this type.  Combining (\ref{eqn-noJ'}) and (\ref{eqn-J'}) we find that
\[
\hom(G,H) \leq s(\delta,H)\Delta^{n-\delta} \bigg( \frac{\sum_{v \in V(H)} d(v)}{\Delta^2} \bigg)^{k-\delta} (1+r_2(\delta,H))
\]
where
\[
r_2(\delta,H) = \frac{1}{s(\delta,H)}  \bigg( \frac{\sum_{v \in V(H)} d(v)}{\Delta} \bigg)^{\delta} \left( \frac{\Delta-1}{\Delta} \right)^{ \frac{n-3C_{H}\delta}{C_{H}\delta-\delta+1}}
\]
%\begin{eqnarray*}
%\hom(G,H) &\leq& s(\delta,H) \Delta^{n-\delta} \Bigg[ \bigg( \frac{\sum_{v \in V(H)} d(v)}{\Delta^{2}} \bigg)^{k-\delta} \\
%&& \qquad + \quad   \Bigg] \\
%&<& s(\delta,H) \Delta^{n-\delta}
%\end{eqnarray*}
For $\delta+1 \leq k \leq C_H\delta$ and $n > c_H \delta^2$, this is smaller than $s(\delta,H) \Delta^{n-\delta}$.

%========================================================================
%Section \delta = 1 proof
%========================================================================

%
%%%%%%%%%%%%%%%%%%%%%%%%%%%%%%%%%%%%%%%%%%%%New Section

\section{Proof of Theorems \ref{thm-Strong delta1} and \ref{thm-max degree condition} ($\delta=1$)}\label{sec-delta1}

While Theorem \ref{thm-max degree condition} is more general than Theorem \ref{thm-Strong delta1}, for clarity of presentation we begin with the proof of Theorem \ref{thm-Strong delta1} and later show how to modify the proof to obtain Theorem \ref{thm-max degree condition}.  Recall that we will assume that $|V(H)| = q$, and furthermore we will assume that $G$ is edge-min-critical, so $G$ has no edge between two vertices of degree larger than one.  (This will give us the inequalities desired; we will address the uniqueness statements in the theorems separately.)  In particular, $G$ is the disjoint union of stars, so we can write $G = \cup_{i} K_{1,n_i-1}$, where $\sum_{i} n_i = n$. %See Figure \ref{fig-stars}.

%\begin{figure}[ht]
%\begin{center}
%\begin{tikzpicture}[scale=1]
%	\node (v1) at (0.5,1) [circle,draw] {};
%	\node (v2) at (0,2) [circle,draw] {};
%	\node (v3) at (1,2) [circle,draw] {};
%%	\node (v4) at (2,2) [circle,draw] {};
%	
%	\node (w1) at (3,1) [circle,draw] {};
%	\node (w2) at (3,2) [circle,draw] {};
%	
%	\node (v5) at (5.5,1) [circle,draw] {};
%	\node (v6) at (6,2) [circle,draw] {};
%	\node (v7) at (5,2) [circle,draw] {};
%	
%	\node (v8) at (9.5,1) [circle,draw] {};
%	\node (v9) at (8,2) [circle,draw] {};
%	\node (v10) at (9,2) [circle,draw] {};
%	\node (v11) at (10,2) [circle,draw] {};
%	\node (v12) at (11,2) [circle,draw] {};
%
%
%	\foreach \from/\to in {v1/v2,v1/v3,v5/v6,v5/v7,v8/v9,v8/v10,v8/v11,v8/v12,w1/w2}
%	\draw (\from) -- (\to);
%	
%\end{tikzpicture}
%\caption{A graph $G$ that is the disjoint union of stars; here we may take $n_1 = 3$, $n_2 = 2$, $n_3 = 3$, and $n_4 = 5$.}
%\label{fig-stars}
%\end{center}
%\end{figure}

Since the stars $K_{1,n_i-1}$ are disjoint and can therefore be colored independently, we have 
\[
\hom(G,H) = \prod_{i} \hom(K_{1,n_i-1},H) = \prod_{i} \hom(K_{1,n_i-1},H)^{\frac{n_i}{n_i}}.
\]
If $x$ is an integer value in $[2,n]$ that maximizes $\hom(K_{1,x-1},H)^{\frac{1}{x}}$, then
\begin{equation}\label{eqn-delta=1}
\hom(G,H) \leq \prod_{i} \hom(K_{1,x-1},H)^{\frac{n_i}{x}} =  \hom(K_{1,x-1},H)^{\frac{n}{x}},
\end{equation}
with equality occurring (when $x|n$) for $\frac{n}{x} K_{1,x-1}$.
Because of this, it will be useful to know the integer value(s) of $x \geq 2$ that maximize $\hom(K_{1,x-1},H)^{\frac{1}{x}}$.   %This makes each vertex in the star $K_{1,x-1}$ have the largest (geometric) average number of choices of color from among all stars.  
%We will show that the maximum occurs, for any $H$ and any $n \geq 2$, at either $x=2$ or $x=n$ and also analyze any cases of equality. 
\medskip

First we derive a formula for $\hom(K_{1,x-1},H)^{\frac{1}{x}}$ for each integer $x \geq 2$.  Notice that all $H$-colorings of $K_{1,x-1}$ can be obtained by coloring the center of the star with any $v \in V(H)$ and then coloring the leaves (independently) with any neighbor of $v$, so
\begin{equation} \label{eqn-star formula}
\hom(K_{1,x-1},H) = \sum_{v \in V(H)} d(v)^{x-1}.
\end{equation}

Now (\ref{eqn-star formula}) holds for each integer $x \geq 2$, and so it will be useful to know the maximum value of 
\begin{equation} \label{eqn-star formula2}
\left( \sum_{v \in V(H)} d(v)^{x-1} \right)^{\frac{1}{x}}
\end{equation}
 over all integers $x \geq 2$.  In fact, we will study (\ref{eqn-star formula2}) in a slightly more general setting; for the remainder of this proof we will analyze (\ref{eqn-star formula2}) over all \emph{real} numbers $x \geq 2$.

\medskip

Recall that we are assuming that $H$ has no isolated vertices, so for all $v \in V(H)$ we have $1 \leq d(v) \leq \Delta$.  Since there exists a $w \in V(H)$ with $d(w) = \Delta$,  we have 
\begin{equation}\label{eqn-asymptotics delta 1}
\left( \sum_{v \in V(H)} d(v)^{x-1} \right)^{\frac{1}{x}} \to \Delta \qquad \text{as }  x \to \infty.
\end{equation}

To obtain more information, for a fixed real $x \geq 2$ let $ a = a(x,H) \in \mathbb{R}$ be such that
\[
d(v_1)^{x-1} + \cdots + d(v_q)^{x-1} = a^x.
\]
Since $1 \leq d(v) \leq \Delta$ for all $v \in V(H)$, for any $\varepsilon > 0$ we have
\begin{equation} \label{eqn-eqn to analyze}
d(v_1)^{x-1 + \varepsilon} + \cdots + d(v_q)^{x-1+\varepsilon} \leq \Delta^{\varepsilon} \left( d(v_1)^{x-1} + \cdots + d(v_q)^{x-1} \right) = \Delta^{\varepsilon} a^{x},
\end{equation}
with strict inequality if $d(v_i) < \Delta$ for some $i$.  Therefore for any $\varepsilon > 0$ (\ref{eqn-eqn to analyze}) gives
\begin{equation}\label{eqn-decreasing delta 1}
a > \Delta \quad \implies \quad \left( \sum_{v \in V(H)} d(v)^{x-1+\varepsilon} \right)^{\frac{1}{x+\varepsilon}} < \left( \sum_{v \in V(H)} d(v)^{x-1} \right)^{\frac{1}{x}}.
\end{equation}
If $a = \Delta$ and $d(v_i) = \Delta$ for all $i$, then (\ref{eqn-eqn to analyze}) gives $q\Delta^{x-1} = \Delta^{x}$ so $H = K_{q}^{\text{loop}}$.  If $a=\Delta$ and $d(v_i) < \Delta$ for some $i$, then for any $\varepsilon > 0$ (\ref{eqn-eqn to analyze}) gives
\begin{equation} \label{eqn-equal delta 1}
\left( \sum_{v \in V(H)} d(v)^{x-1+\varepsilon} \right)^{\frac{1}{x+\varepsilon}} < \Delta = \left( \sum_{v \in V(H)} d(v)^{x-1} \right)^{\frac{1}{x}}.
\end{equation}
  
Finally, for any $\varepsilon > 0$, if $a < \Delta$ then $\Delta^\varepsilon a^x < \Delta^{x + \varepsilon}$, and so  (\ref{eqn-eqn to analyze}) gives
\begin{equation}\label{eqn-increasing delta 1}
a < \Delta \quad \implies \quad \left( \sum_{v \in V(H)} d(v)^{x-1+\varepsilon} \right)^{\frac{1}{x+\varepsilon}} < \Delta.
\end{equation}

This already provides a substantial amount of information, fully analyzing the graphs $H$ where $\sum_{v \in V(H)} d(v) \geq \Delta^2$ (here we focus on $x=2$ and so the condition on $H$ means $a=a(2,H) \geq \Delta$).  Indeed, for $H \neq K_{q}^{\text{loop}}$, (\ref{eqn-decreasing delta 1}) and (\ref{eqn-equal delta 1}) applied at $x=2$ imply 
\[
\left( \sum_{v \in V(H)} d(v)^{y-1} \right)^{\frac{1}{y}} < \left( \sum_{v \in V(H)} d(v) \right)^{\frac{1}{2}}
\]
for any $y > 2$, which implies
$
\hom(K_{1,x-1},H)^{\frac{1}{x}} < \hom(K_{1,1},H)^{\frac{1}{2}}
$
for any integer $x > 2$.

For the graphs $H$ satisfying $\sum_{v \in V(H)} d(v) < \Delta^2$, we may already obtain a statement for large $n$ (using (\ref{eqn-asymptotics delta 1}) and also (\ref{eqn-increasing delta 1}) at $x=2$), but with an additional argument we can obtain a statement for all $n$.  We need the following lemma, whose proof we give after first using the lemma to complete the proof of Theorem \ref{thm-Strong delta1}. 

\begin{lemma}\label{lem-lp delta 1}
The function $\left(\sum_{v \in V(H)} d(v)^{x-1}\right)^{\frac{1}{x}}$ %, viewed as a function of $x \in (2,\infty)$, 
has at most one local maximum or minimum.
\end{lemma}
If we assume Lemma \ref{lem-lp delta 1}, then (\ref{eqn-asymptotics delta 1}), (\ref{eqn-decreasing delta 1}), (\ref{eqn-equal delta 1}), and (\ref{eqn-increasing delta 1}) show that for $H \neq K_{q}^{\text{loop}}$ the function $\left(\sum_{v \in V(H)} d(v)^{x-1}\right)^{\frac{1}{x}}$ is either decreasing to $\Delta$ on $(2,\infty)$, increasing to $\Delta$ on $(2,\infty)$, or decreasing on $(2,x_0)$ and increasing to $\Delta$ on $(x_0,\infty)$ for some $x_0 > 2$.  See Figure \ref{fig-delta1possibilities} for the possible behaviors of $\left(\sum_{v \in V(H)} d(v)^{x-1}\right)^{\frac{1}{x}}$. So if $H \neq K_{q}^{\text{loop}}$, then (\ref{eqn-delta=1}) shows that Theorem \ref{thm-Strong delta1} holds for any edge-min-critical $G \in \calG(n,1)$. This implies that the upper bounds given in Theorem \ref{thm-Strong delta1} hold for any $G \in \calG(n,1)$.

\begin{center}
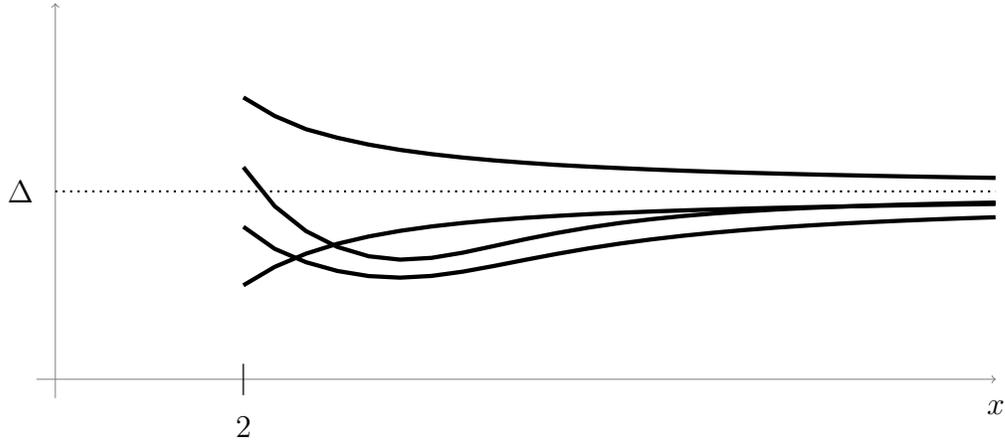
\begin{figure}[ht]

\begin{tikzpicture}[scale=1.25]
\draw [help lines, ->] (-.2,0) -- (10,0);
\draw [help lines, ->] (0,-.2) -- (0,4);
\draw [thick,dotted] (0,2) -- (10,2);
\node at (0,2) [label=180:$\Delta$] {};

\node at (10,0) [label=270:$x$] {};

%\draw [thick,dotted] (2,0) -- (2,4);
\node at (2,0) [label=270:$2$] {$|$};
%\node at (3.6503,0) [label=270:$x_0$] {$|$};

\draw [ultra thick,domain=2:10] plot (\x, {pow(pow(2,\x+1) + 1,(1/\x))});
\draw [ultra thick,domain=2:10] plot (\x, {pow(pow(2,(\x-.05)-2) -2*(\x-.05) + 5.5,(1/\x))});
\draw [ultra thick,domain=2:10] plot (\x, {pow(pow(2,\x-0.8) -4.6*\x + 12,(1/\x))});
\draw [ultra thick,domain=2:10] plot (\x, {pow(pow(2,\x-1) - 1,(1/\x))});

\end{tikzpicture}
\caption{The possible behaviors of the function $\left(\sum_{v \in V(H)} d(v)^{x-1}\right)^{\frac{1}{x}}$ for $H \neq K_{\Delta}^{\text{loop}}$.}
\label{fig-delta1possibilities}
\end{figure}
\end{center}

Finally, we need to argue that for $H\neq K_{q}^{\text{loop}}$, the edge-min-critical graphs in $\calG(n,1)$ which achieve equality are the only possible graphs in $\calG(n,1)$ which achieve equality.  It suffices to consider the addition of a single edge to one of the graphs achieving equality and showing that the number of $H$-colorings decreases in this case.

By considering the neighbors of a vertex $v \in V(H)$ with $d(v) = \Delta$, adding any edge to a disjoint union of stars strictly lowers the number of $H$-colorings unless $H$ contains $K_{\Delta}^{\text{loop}}$ (a slight modification of the argument given in Case 1 of Section \ref{sec-deltageneral} will work, realizing that we need to consider both edges joining vertices in the same component and also edges joining vertices in different components).  If $H$ does contain $K_{\Delta}^{\text{loop}}$ and $H \neq K_{\Delta}^{\text{loop}}$, then $H$ contains some other component and furthermore $H$ clearly satisfies $\sum_{v \in V(H)} d(v) \geq \Delta^2$.  Since $\frac{n}{2} K_2$ is the unique edge-min-critical graph achieving equality for this $H$ and $H$ has at least $2$ components, adding any edge to $\frac{n}{2} K_2$ (which will necessarily join together two components of $\frac{n}{2} K_2$) will lower the number of $H$-colorings in this case as well.  We leave the details to the reader.  This completes the proof of Theorem \ref{thm-Strong delta1}.

\begin{proof}[Proof of Lemma \ref{lem-lp delta 1}:]
\textit{} This lemma is a corollary of the following proposition about $L^p$ norms, which is a special case of Lemma 1.11.5 in \cite{Tao} (or, equivalently, Lemma 2 in Terence Tao's blog post 245C, Notes 1: Interpolation of $L^p$ spaces).  Recall that we assume $H$ has no isolated vertices.   

\begin{proposition}\label{thm-Lp}
Define a measure $\mu$ on $V(H)$ by $\mu(v) = \frac{1}{d(v)}$, and let $g:V(H) \to \mathbb{R}$ be given by $g(v) = d(v)$.  Then the function defined by $\frac{1}{x} \mapsto ||g||_{L^{x}(V(H))} = \left( \sum_{v \in V(H)} d(v)^{x-1} \right)^{\frac{1}{x}}$ %\left( \sum_{v} d(v)^{x-1} \right)^{1/x}$ 
is log-convex for $x \in (2, \infty)$.
\end{proposition}

Recall that a log-convex function can has most one local maximum or local minimum.  The composition of the reciprocal map and the map given in Proposition \ref{thm-Lp} is the function defined by $x \mapsto \left( \sum_{v \in V(H)} d(v)^{x-1} \right)^{\frac{1}{x}}$.  Since the reciprocal is strictly monotone and therefore preserves local extremal values, Lemma \ref{lem-lp delta 1} follows. 
\end{proof}

\medskip

Lastly, we prove Theorem \ref{thm-max degree condition}.  Notice that we can still delete any edge from a graph $G \in \calG(n,\delta,D)$ and remain in $\calG(n,\delta,D)$ as long the edge deletion does not lower the minimum degree.  Therefore the proof of Theorem \ref{thm-Strong delta1} also proves the inequality in Theorem \ref{thm-max degree condition} by restricting the function $\left(\sum_{v \in V(H)} d(v)^{x-1} \right)^{\frac{1}{x}}$ to values in $[2,D+1]$.  

Having dealt with the inequality, we need to argue that the edge-min-critical graphs in $\calG(n,1,D)$ which achieve equality are the only possible graphs in $\calG(n,1,D)$ which achieve equality.  This follows from similar arguments to those in the proof of Theorem \ref{thm-Strong delta1} and we again leave the details to the reader.

%%=======================================================================
%%Section: 2 regular graphs
%%=======================================================================

\section{Proof of Theorem \ref{thm-2 regular graphs}} \label{sec-2regular}

%%%%%%%%%%%%%%%%%%%%%%%%%%%%%%%%%%%%%%%%%%%%%%%%%%%%New Section

Recall that we will assume that $|V(H)| = q$, and we begin with a few remarks about the number of $H$-colorings of a cycle $C_k$.  Let $A$ denote the adjacency matrix of $H$.  Then for $k \geq 3$, $\hom(C_k,H) = \Tr A^k$; indeed, the diagonal entry $(ii)$ in $A^k$ counts the number of $H$-colorings of the path on $k+1$ vertices $P_{k+1}$ that color both endpoints with color $i$, and by identifying the endpoints we obtain a coloring of $C_k$ with one fixed vertex having color $i$.  Therefore if $\lambda_1, \lambda_2, \ldots, \lambda_q$ are the eigenvalues of $A$, then 
\begin{equation}\label{eqn-H col cycles}
\hom(C_k,H) = \lambda_1^k + \cdots + \lambda_q^k.
\end{equation}
It is possible to obtain results using ideas based on Proposition \ref{thm-Lp} (with some additional observations); we provide an alternate proof.  Without loss of generality, assume that $\lambda_1 \geq \lambda_2 \geq \cdots \geq \lambda_q$.  Notice that $\lambda_1 > 0$ and $\lambda_1 \geq |\lambda_q|$; this follows from the Perron-Frobenius theorem, %(see e.g. \cite{Seneta}), 
but is also immediate as otherwise (\ref{eqn-H col cycles}) would imply that $\hom(C_k,H) < 0$ for large odd $k$.  %Also, $\lambda_1 \geq 1$, since otherwise (\ref{eqn-H col cycles}) would imply that eventually $\hom(C_k,H) < 1$ for even $k$, but $C_k$ can map to the endpoints of any edge in $H$.

\medskip

We first address the inequality in Theorem \ref{thm-2 regular graphs}, and deal with the cases of equality at the end.  Suppose $k \geq 4$ is even and let $b = b(k,H) \geq \lambda_1$ be such that
\[
\lambda_1^k + \lambda_2^k + \cdots + \lambda_q^k = b^k.
\]
Then 
\[
\lambda_1^{k+2} + \cdots + \lambda_q^{k+2} \leq \lambda_1^2 (\lambda_1^k + \cdots + \lambda_q^k) \leq b^2 (\lambda_1^k + \cdots + \lambda_q^k) = b^{k+2},
\]
with equality only for $b = \lambda_1$, so 
\begin{equation}\label{eqn-cycle even}
(\lambda_1^{k+2} + \cdots + \lambda_q^{k+2})^{\frac{1}{k+2}} \leq (\lambda_1^{k} + \cdots + \lambda_{q}^k)^{\frac{1}{k}},
\end{equation}
which implies that $\hom(C_k,H)^{\frac{1}{k}} \leq \hom(C_4,H)^{\frac{1}{4}}$ for even $k \geq 6$.

Suppose next that $k \geq 5$ is odd, and so as above we have $\lambda_1^{k-1} + \cdots + \lambda_q^{k-1} = b^{k-1}$.  Then
\begin{equation}\label{eqn-cycle odd}
\lambda_1^k + \cdots + \lambda_q^k \leq |\lambda_1|\lambda_1^{k-1} + \cdots + |\lambda_q|\lambda_q^{k-1} \leq \lambda_1 b^{k-1} \leq b^k,
\end{equation}
with equality only for $b = \lambda_1$, which implies that $\hom(C_k,H)^{\frac{1}{k}} \leq \hom(C_{k-1},H)^{\frac{1}{k-1}}$.  

Summarizing the above, the function $\hom(C_k,H)^{\frac{1}{k}}$ is non-increasing from every even $k\geq 4$ to both $k+1$ and $k+2$ and so, for $k \geq 5$, $\hom(C_k,H)^{\frac{1}{k}} \leq \hom(C_4,H)^{\frac{1}{4}}$.  Therefore for all $k \geq 3$,
\[
\hom(C_k,H)^{\frac{1}{k}} \leq \max \{ \hom(C_3,H)^{\frac{1}{3}}, \hom(C_4,H)^{\frac{1}{4}} \}.
\]

Now, if $G$ is any $2$-regular graph, then $G$ is the disjoint union of cycles $C_{k_i}$.  So if $\hom(C_4,H)^{\frac{1}{4}} \geq \hom(C_3,H)^{\frac{1}{3}}$,
\[
\hom(G,H) = \prod_i \hom(C_{k_i},H) \leq \prod_i \hom(C_4,H)^{\frac{k_i}{4}} = \hom(C_4,H)^{\frac{n}{4}},
\]
with a similar statement holding if $\hom(C_3,H)^{\frac{1}{3}} \geq \hom(C_4,H)^{\frac{1}{4}}$.

Finally, we deal with the cases of equality.  There is equality in (\ref{eqn-cycle even}) and (\ref{eqn-cycle odd}) only when $b = \lambda_1$ (and so $\lambda_2 = \cdots = \lambda_q = 0$).  Recall that $A$ is symmetric and so has distinct eigenvectors associated to each $\lambda_i$, so in this case $A$ has rank 1 and therefore all rows of $A$ are scalar multiples of any other row.  If any entry $A_{(ij)} = 0$, then some column and row of $A$ is the $0$ vector, which corresponds to an isolated vertex in $H$.  Since we assume $H$ has no isolated vertices, $A$ must be the matrix of all $1$'s and so $H=K_{q}^{\text{loop}}$.  Therefore, for $H \neq K_{q}^{\text{loop}}$ we have $\hom(C_k,H)^{\frac{1}{k}} < \max\{\hom(C_3,H)^{\frac{1}{3}}, \hom(C_4,H)^{\frac{1}{4}} \}$ whenever $k \geq 5$.  The statement about equality is now evident.

%========================================================================
%Section delta=2
%========================================================================

\section{Proof of Theorem \ref{thm-Strong delta2} ($\delta=2$)}\label{sec-delta2}

\subsection{Preliminary remarks}

We first gather together a number of observations that we will use in the proof.  The following lemma is easily adapted from Lemma 3.2 in \cite{EngbersGalvin3}.  Recall than a graph with minimum degree $\delta$ is edge-min-critical if deleting any edge reduces the minimum degree of the resulting graph.

\begin{lemma}\label{lem-delta2structure}
Let $G \in \calG(n,2)$ be an edge-min-critical graph.  Either
\begin{enumerate}
	\item $G$ is a disjoint union of cycles, or
	
	\item $V(G)$ may be partitioned into $Y_1 \cup Y_2$, with $1 \leq |Y_1| \leq n-3$ in such a way that $Y_1$ induces a path, $Y_2$ induces a graph with minimum degree $2$, each endpoint of the path induced by $Y_1$ has exactly one edge to $Y_2$, the endpoints of these two edges to $Y_2$ are either the same or non-adjacent, and there are no other edges from $Y_1$ to $Y_2$ (see Figure \ref{fig-lemma}).
\end{enumerate}
\end{lemma}

%In the second case, $Y_1$ is a path on $|Y_1|$ vertices, which will be important when we iteratively build $G$ in Corollary \ref{cor-delta2decomposition}.  
In the statement given in \cite{EngbersGalvin3}, the graph is assumed to be connected and \emph{min-critical}, i.e. both edge-min-critical and \emph{vertex-min-critical}.  Being vertex-min-critical --- meaning the deletion of any vertex reduces the minimum degree --- is only used to show that $|Y_1| \geq 2$; the result holds for all edge-min-critical graphs by relaxing to $|Y_1| \geq 1$.  To obtain the non-connected version, simply apply the connected version to each component.  If some component is not a cycle, then augment the $Y_2$ obtained with the vertices in every other component.

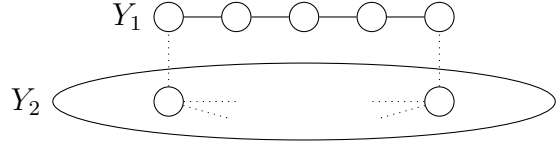
\begin{figure}[ht]
\begin{center}
\begin{tikzpicture}[scale=.9]
	\node (v1) at (1,1) [circle,draw] {};
	\node (v2) at (5,1) [circle,draw] {};
	
	\node [ellipse,draw, fit=(v1) (v2),label=180:$Y_2$] {};
	
	\node (w1) at (1,2.25) [circle,draw,label=180:$Y_1$] {};
	\node (w2) at (2,2.25) [circle,draw] {};
	\node (w3) at (3,2.25) [circle,draw] {};
	\node (w4) at (4,2.25) [circle,draw] {};
	\node (w5) at (5,2.25) [circle,draw] {};
	
	\foreach \from/\to in {w1/w2,w2/w3,w3/w4,w4/w5}
	\draw (\from) -- (\to);
	
	\draw [dotted] (v1) -- (w1);
	\draw [dotted] (v2) -- (w5);
	\draw [dotted] (v1) -- (2,1);
	\draw [dotted] (v1) -- (1.90,.75);
	\draw [dotted] (v2) -- (4,1);
	\draw [dotted] (v2) -- (4.10,.75);

\end{tikzpicture}
\end{center}
\caption{A example of a path on $5$ vertices ($|Y_1|=5$) given in Lemma \ref{lem-delta2structure}.}
\label{fig-lemma}
\end{figure}

\begin{corollary}\label{cor-delta2decomposition}
Let $G \in \calG(n,2)$ be an edge-min-critical graph.  Then $G$ may be constructed via the following iterative procedure:
\begin{itemize}
	\item Start with a non-empty collection of disjoint cycles.
	\item Next, iteratively add a %(possibly empty) 
	collection paths on $k \geq 2$ vertices which connect to existing vertices of the graph \emph{only} at the endpoints of the path; let $G'$ be the graph obtained after these paths have been added.
	\item Finally, add $n-|V(G')|$ new vertices, each of which is adjacent to exactly two vertices of $G'$.
	%Finally, add a collection of paths on $1$ vertex (in other words, isolated vertices), each of which connects to exactly two distinct existing vertices of the graph, all at the same time.
\end{itemize}
\end{corollary}

We emphasize that when we add a path on $k$ vertices to a graph (via the construction in Corollary \ref{cor-delta2decomposition}), we add $k$ \emph{new} vertices to the graph that induce a path in the augmented graph, so that the endpoints of the path are each adjacent to one existing vertex and no other vertices of the path are adjacent to existing vertices.  Figure \ref{fig-lemma} may be viewed as an example of adding a path on $5$ vertices ($Y_1$) to the graph on vertex set $Y_2$.

\begin{proof}
By Lemma \ref{lem-delta2structure}, paths on $k \geq 1$ vertices can be removed inductively until a collection of disjoint cycles remain, and so we may construct any graph $G$ starting with the cycles.  Reversing this, we may iteratively add paths on $k \geq 1$ vertices to produce $G$.  The content of this corollary is that $G$ may be constructed by adding all paths on $k \geq 2$ vertices before the paths on $1$ vertex, and the paths on $1$ vertex may all be added at the same time.

Why is this possible?  Adding a path on $1$ vertex creates a vertex of degree $2$ adjacent to two vertices of degree at least $3$.  Since $G$ is edge-min-critical, no future path on $k$ vertices will be adjacent to the vertex of degree $2$.
\end{proof}

Lemma \ref{lem-delta2structure} is enough to prove the case when $\max\{ \hom(C_3,H)^{\frac{1}{3}}, \hom(C_4,H)^{\frac{1}{4}} \} \geq \Delta$ (without a characterization of uniqueness); we will provide the details of this in the next section.  A graph $H$ which satisfies
\begin{equation}\label{eqn-Hcondition}
\max\{ \hom(C_3,H)^{\frac{1}{3}}, \hom(C_4,H)^{\frac{1}{4}}\} < \Delta
\end{equation}
requires a few more observations.

\begin{lemma}\label{lem-color p4}
For any two vertices $u,v$ of $H$ (not necessarily distinct), there are at most $\Delta^2$ $H$-colorings of $P_4$ that map the initial vertex of the path to $u$ and the terminal vertex to $v$. If $H$ does not contain $K_{\Delta}^{\text{loop}}$ or $K_{\Delta,\Delta}$ as a component, then there are strictly fewer than $\Delta^2$ such $H$-colorings.
 
\end{lemma}

\begin{proof}
The first statement is obvious, since $P_4$ is connected and the maximum degree of $H$ is $\Delta$.  Suppose there are $\Delta^2$ extensions to an $H$-coloring of $P_4$.  Let $P_4$ have vertices $w_1$ through $w_4$ and edges $w_1 \sim w_2$, $w_2 \sim w_3$, and $w_3 \sim w_4$, and let $w_1$ and $w_4$ be given colors $u$ and $v$ in $H$, respectively.  We color $w_2$ first (with color $v_2$) and then $w_3$ (with color $v_3$), conditioning on whether $u$ is looped or not.

Suppose that $u$ is unlooped in $H$.  Clearly $d(u) = \Delta$ and each neighbor of $u$ also has degree $\Delta$.  Since some of the $\Delta^2$ extensions map $w_3$ to $u$, it must be the case that $v \sim_H u$.  Furthermore, as $w_2$ maps to $v_2$ (so necessarily $v_2 \sim_H u$), every neighbor of $v_2$ must be adjacent to $v$.  Since $v_2$ can be any neighbor of $u$, this implies that $K_{\Delta,\Delta}$ is the component of $H$ containing $u$ and $v$.

A similar analysis for looped $u$ shows that $K_{\Delta}^{\text{loop}}$ is the component of $H$ containing $u$ and $v$.
\end{proof}

\begin{corollary}\label{cor-short paths}
Suppose that $H$ satisfies (\ref{eqn-Hcondition}) and let $k \geq 4$. For any two vertices $u,v$ of $H$ (not necessarily distinct), there are strictly fewer than $\Delta^{k-2}$ $H$-colorings of $P_k$ that map the initial vertex of the path to $u$ and the terminal vertex to $v$.
\end{corollary}

\begin{proof}
Notice that $\hom(C_4,K_{\Delta,\Delta})^{\frac{1}{4}} > \Delta$ and $\hom(C_4, K_{\Delta}^{\text{loop}})^{\frac{1}{4}} = \Delta$.  Color, beginning from one endpoint, until there are two uncolored vertices left.  Then apply Lemma \ref{lem-color p4}. %Each of these vertices has at most $\Delta$ choices for a color.  Combining with fewer than $\Delta^2$ choices for the colors on the final two vertices gives the result.
\end{proof}

We can strengthen Corollary \ref{cor-short paths} when $k$ is large.

\begin{lemma}\label{lem-long paths}
Suppose that $H$ satisfies (\ref{eqn-Hcondition}).  Then there exists a constant $l_H$ (depending on $H$) such that if $k \geq l_H$ and the endpoints of $P_k$ are mapped to $H$, then there are strictly fewer than $\frac{1}{|V(H)|^2} \Delta^{k-4}$ extensions to an $H$-coloring of $P_k$.
\end{lemma}

\begin{proof}
Notice that a path must be mapped to a connected component of $H$; focus on that component.  If $A$ is the adjacency matrix of that component, then the number of $H$-colorings of $P_k$ with endpoints colored $i$ and $j$ is $A^{k}_{(ij)}$.  If $\lambda_1$ denotes the largest eigenvalue of $A$, then by the Perron-Frobenius Theorem %(see for example \cite[Theorem 1.5]{Seneta}) 
there exists a strictly positive vector $\textbf{x}$ such that $A^k \textbf{x} = \lambda_1^k \textbf{x}$ for all $k\geq 1$.  By considering the row of $A$ containing $\max_{i,j} A^k_{(ij)}$, we see that there is a constant $c$ such that $\max_{i,j} A^k_{(ij)} \leq c \lambda_1^k$ (we can take $c = \max_j x_j / \min_j x_j$, where $\textbf{x} = (x_j)$).  Since $\lambda_1 \leq \left( \sum_i \lambda_i^4 \right)^{\frac{1}{4}} = \hom(C_4,H)^{\frac{1}{4}} < \Delta$ implies $\lambda_1 < \Delta$, this proves the lemma.

\end{proof}

\subsection{The proof}

We are now ready to prove Theorem \ref{thm-Strong delta2}. %Recall that our goal is to show that for any $H$ and any $G \in \calG(n,2)$,
%\[
%\hom(G,H) \leq \max \{\hom(C_3,H)^{n/3}, \hom(C_4,H)^{n/4}, \hom(K_{2,n-2},H) \}.
%\]
We assume that $G$ is edge-min-critical until we discuss the cases of equality in the upper bound.  First, suppose that $H$ satisfies 
\begin{equation}\label{eqn-delta2Hcondition1}
\max \{\hom(C_3,H)^{\frac{1}{3}}, \hom(C_4,H)^{\frac{1}{4}} \} \geq \Delta.
\end{equation}
Using induction on $n$, we will show that for any $G \in \calG(n,2)$,
\[
\hom(G,H) \leq \max \{ \hom(C_3,H)^{\frac{n}{3}}, \hom(C_4,H)^{\frac{n}{4}} \}.
\]
The base case $n=3$ is trivial.

For the inductive step, assume first that $\hom(C_3,H)^{\frac{1}{3}} \leq \hom(C_4,H)^{\frac{1}{4}}$.  If all components of $G$ are cycles, then we are finished by Theorem \ref{thm-2 regular graphs}.  If some component of $G$ is not a cycle, then by Lemma \ref{lem-delta2structure} we can partition $V(G)$ into $Y_1 \cup Y_2$, with $1 \leq |Y_1| \leq n-3$ and $Y_1$ connected to $Y_2$.  We imagine first coloring $Y_2$ and then extending this to $Y_1$.  By induction, there are at most $\hom(C_4,H)^{\frac{n-|Y_1|}{4}}$ $H$-colorings of $Y_2$.  But since $Y_1$ is connected to $Y_2$, for every fixed $H$-coloring of $Y_2$, each vertex in $Y_1$ has at most $\Delta$ choices for a color.  Therefore,
\begin{equation}\label{eqn-C4 wins}
\hom(G,H) \leq \Delta^{|Y_1|} \hom(C_4,H)^{\frac{n-|Y_1|}{4}} \leq \hom(C_4,H)^{\frac{n}{4}}.
\end{equation}
The case when $\hom(C_3,H)^{\frac{1}{3}} \geq \hom(C_4,H)^{\frac{1}{4}}$ is similar.  

With the upper bound established in this case, we turn to the cases of equality.  First suppose that $H$ satisfies $\max \{ \hom(C_3,H)^{\frac{1}{3}}, \hom(C_4,H)^{\frac{1}{4}}\} > \Delta$.  Then (\ref{eqn-C4 wins}) is strict, which implies that equality can only be obtained for the disjoint union of cycles and hence Theorem \ref{thm-2 regular graphs} provides the cases of equality among edge-min-critical graphs.  

Now suppose $\max \{ \hom(C_3,H)^{\frac{1}{3}}, \hom(C_4,H)^{\frac{1}{4}} \} = \Delta$ (which implies that $H$ cannot have $K_{\Delta,\Delta}$ as a component).  Notice that equality is achieved for any $G$ with $\hom(G,H) = \Delta^n$, and suppose that $H \neq K_{\Delta}^{\text{loop}}$ (so since $\hom(C_4,H)^{\frac{1}{4}} \leq \Delta$, $H$ cannot contain $K_{\Delta}^{\text{loop}}$ as a component).  By Theorem \ref{thm-2 regular graphs}, the construction of $G$ in Corollary \ref{cor-delta2decomposition} must start with disjoint copies of $C_3$ and/or $C_4$.  Corollary \ref{cor-short paths} implies that only paths on $1$ vertex may be added to these cycles, and to achieve the bound of $\Delta^n$, \emph{every} coloring of these cycles must provide $\Delta$ choices for the color of the vertex in the path on $1$ vertex.  We outline the possible situations which occur when adding a path on $1$ vertex to the cycles in Figure \ref{fig-annoyingcases}; the vertex labeled $v$ must have $\Delta$ choices for a color regardless of how the adjacent cycles are colored.

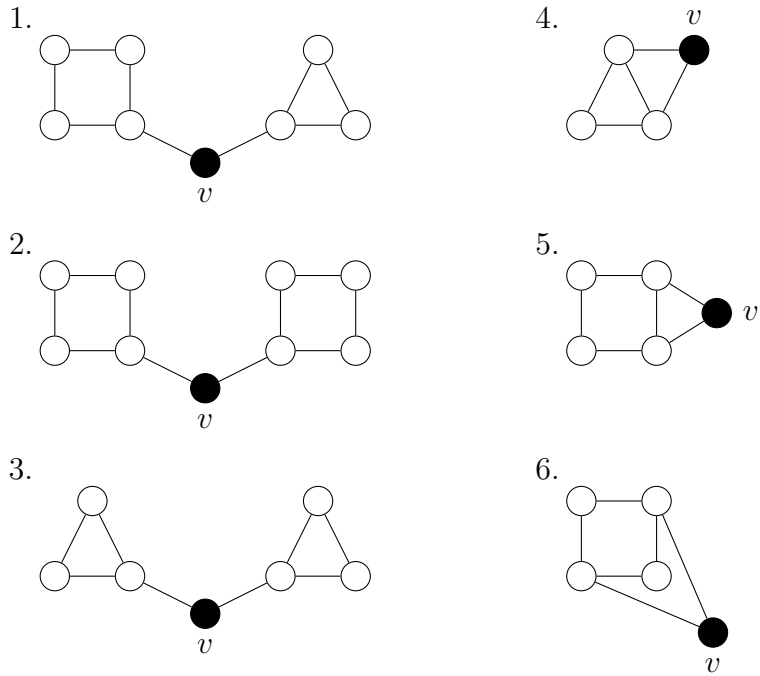
\begin{figure}[ht] 
\begin{center}
\begin{tikzpicture}
	\node (a1) at (10,10) [circle,draw] {};
	\node (a2) at (10,11) [circle,draw,label=135:$1.$] {};
	\node (a3) at (11,10) [circle,draw] {};
	\node (a4) at (11,11) [circle,draw] {};
	\node (a5) at (13,10) [circle,draw] {};
	\node (a6) at (14,10) [circle,draw] {};
	\node (a7) at (13.5,11) [circle,draw] {};
	\node (v1) at (12,9.5) [circle,draw,fill,label=270:$v$] {};
	
	\foreach \from/\to in {a1/a2,a2/a4,a4/a3,a3/a1,a5/a6,a6/a7,a7/a5,v1/a3,v1/a5}
	\draw (\from) -- (\to);
	
	\node (b1) at (10,7) [circle,draw] {};
	\node (b2) at (10,8) [circle,draw,label=135:$2.$] {};
	\node (b3) at (11,7) [circle,draw] {};
	\node (b4) at (11,8) [circle,draw] {};	
	\node (b5) at (13,7) [circle,draw] {};
	\node (b6) at (13,8) [circle,draw] {};
	\node (b7) at (14,7) [circle,draw] {};
	\node (b8) at (14,8) [circle,draw] {};
	\node (v2) at (12,6.5) [circle,draw,fill,label=270:$v$] {};
	
	\foreach \from/\to in {b1/b2,b2/b4,b4/b3,b3/b1,b5/b6,b6/b8,b8/b7,b7/b5,v2/b3,v2/b5}
	\draw (\from) -- (\to);	
	
	\node (c1) at (10,4) [circle,draw] {};
	\node (c2) at (11,4) [circle,draw] {};
	\node at (10,5) [label=135:$3.$] {};
	\node (c3) at (10.5,5) [circle,draw] {};
	\node (c4) at (13,4) [circle,draw] {};
	\node (c5) at (14,4) [circle,draw] {};
	\node (c6) at (13.5,5) [circle,draw] {};
	\node (v3) at (12,3.5) [circle,draw,fill,label=270:$v$] {};
	
	\foreach \from/\to in {c1/c2,c2/c3,c3/c1,c4/c5,c5/c6,c6/c4,v3/c2,v3/c4}
	\draw (\from) -- (\to);	
	
	\node (d1) at (17,10) [circle,draw] {};
	\node at (17,11) [label=135:$4.$] {};
	\node (d2) at (18,10) [circle,draw] {};
	\node (d3) at (17.5,11) [circle,draw] {};
	\node (v4) at (18.5,11) [circle,draw,fill,label=90:$v$] {};
	
	\foreach \from/\to in {d1/d2,d2/d3,d3/d1,v4/d2,v4/d3}
	\draw (\from) -- (\to);
	
	\node (e1) at (17,7) [circle,draw] {};
	\node (e2) at (17,8) [circle,draw,label=135:$5.$] {};
	\node (e3) at (18,8) [circle,draw] {};
	\node (e4) at (18,7) [circle,draw] {};
	\node (v5) at (18.8,7.5) [circle,draw,fill,label=0:$v$] {};
	
	\foreach \from/\to in {e1/e2,e2/e3,e3/e4,e4/e1,v5/e3,v5/e4}
	\draw (\from) -- (\to);
	
	\node (f1) at (17,4) [circle,draw] {};
	\node (f2) at (17,5) [circle,draw,,label=135:$6.$] {};
	\node (f3) at (18,5) [circle,draw] {};
	\node (f4) at (18,4) [circle,draw] {};
	\node (v6) at (18.75,3.25) [circle,draw,fill,label=270:$v$] {};
	
	\foreach \from/\to in {f1/f2,f2/f3,f3/f4,f4/f1,v6/f3,v6/f1}
	\draw (\from) -- (\to);
	
\end{tikzpicture}

\caption{The possible situations which occur when a path on $1$ vertex (labeled $v$) is added to the cycles.}

\end{center}
\label{fig-annoyingcases}
\end{figure}

We will prove that in the situation of Case 5 from Figure \ref{fig-annoyingcases}, having $\Delta$ choices for the color on $v$ for each coloring of $C_4$ forces $H$ to contain $K_{\Delta}^{\text{loop}}$ as a component.  Suppose that the neighbors of $v$ have colors $i$ and $j$.  We can assume that $i \neq j$, since $C_4$ can map its partition classes to the endpoints of any edge in $H$.  Since $v$ has $\Delta$ possibilities for its color, necessarily $i$ and $j$ must each have $\Delta$ neighbors and furthermore those $\Delta$ neighbors must be simultaneously neighbors of both $i$ and $j$.  In particular, since $i$ and $j$ are on adjacent vertices of $C_4$, we have $i \sim_H j$ and so $i$ and $j$ must be possible colors for $v$.  This means that $i$ and $j$ must be looped and that all neighbors of $i$ are also neighbors of $j$.  But if $k$ is any other neighbor of $i$, then a similar argument (replacing $j$ by $k$) shows that $k$ is looped and is adjacent to all other neighbors of $i$.  Therefore the component containing $i$ is $K_{\Delta}^{\text{loop}}$, which contradicts our assumption that $K_{\Delta}^{\text{loop}}$ is not a component of $H$. A routine analysis of the five remaining cases from Figure \ref{fig-annoyingcases} shows that having $\Delta$ choices for the color of the vertex in a path on $1$ vertex always forces $H$ to contain $K_{\Delta}^{\text{loop}}$ as a component; we leave the details of these remaining cases to the reader.  Therefore, equality can only occur when $G$ is a disjoint union of cycles, so Theorem \ref{thm-2 regular graphs} again characterizes the cases of equality among edge-min-critical graphs.

Finally we need to show that edge-min-critical graphs are the only graphs achieving equality.  Arguing as in Case 1 of Section \ref{sec-deltageneral}, adding any edge to a $C_4$ will strictly lower the number of $H$-colorings unless $H$ contains $K_{\Delta}^{\text{loop}}$.  The cases of adding an edge between two disjoint cycles are similar, and so adding any edge to a graph achieving equality will strictly lower the number of $H$-colorings unless $H$ is the disjoint union of some number of fully looped complete graphs.  If $H$ is of this form and $H \neq K_{\Delta}^{\text{loop}}$, then $\hom(C_3,H)^{\frac{1}{3}} > \hom(C_4,H)^{\frac{1}{4}}$, and so in fact adding any edge to $\frac{n}{3} C_3$ will strictly lower the number of $H$-colorings since the disjoint copies of $C_3$ can be colored using different components of $H$, but the copies of $C_3$ joined by an edge must all be colored by a single component of $H$.  %This fully characterizes the cases of equality.

%  A routine but tedious analysis shows that always having $\Delta$ choices for the color of the vertex in a path of length $1$ forces $H$ to contain $K_{\Delta}^{\text{loop}}$ as a component.  Therefore, equality can only occur when $G$ is a disjoint union of cycles, so Theorem \ref{thm-2 regular graphs} again characterizes the cases of equality among edge-critical graphs.

%Finally, adding any edge to a graph achieving equality will strictly lower the number of $H$-colorings unless $H$ is the disjoint union of some number of fully looped complete graphs.  If $H$ is of this form and $H \neq K_{\Delta}^{\text{loop}}$, then $\hom(C_3,H)^{\frac{1}{3}} > \hom(C_4,H)^{\frac{1}{4}}$, and so in fact adding any edge to $\frac{n}{3} C_3$ will strictly lower the number of $H$-colorings.  %This fully characterizes the cases of equality.

\medskip
Now suppose that $H$ satisfies
\begin{equation}\label{eqn-delta2HCondition2}
\max \{ \hom(C_3,H)^{\frac{1}{3}}, \hom(C_4,H)^{\frac{1}{4}} \} < \Delta.
\end{equation}
%We will show that for large enough $n$ and any $G \in \calG(n,2)$,
%\[
%\hom(G,H) \leq \hom(K_{2,n-2},H),
%\]
%with equality only for $G = K_{2,n-2}$.  
Recall that from Theorem \ref{thm-2 regular graphs} we have
\[
\hom(C_k,H)^{1/k} \leq \max \{\hom(C_3,H)^{1/3}, \hom(C_4,H)^{1/4} \},
\]
for all $k \geq 3$, which we bound (for simplicity) by
\begin{equation}\label{eqn-bound on cycles}
\hom(C_k,H) \leq (\Delta^4 - 1)^{k/4} \qquad \text{for } k\geq 3.
\end{equation}
As in the proof of Theorem \ref{thm-ExtremalHColorings}, we will let $S(2,H)$ denotes the vectors in $V(H)^2$  with the property that the elements of the vector have $\Delta$ common neighbors, and $s(2,H) = |S(2,H)|$.  Notice that $\hom(K_{2,n-2}) \geq s(2,H) \Delta^{n-2}$.

\medskip

Suppose first that $G$ is edge-min-critical.  We will utilize the construction of $G$ from Corollary \ref{cor-delta2decomposition} to produce all $H$-colorings of $G$ by coloring the disjoint cycles first and then coloring the paths.

If there are more than $t$ vertices in the disjoint cycles, then by (\ref{eqn-bound on cycles}) we have $\hom(G,H) \leq (\Delta^4-1)^{t/4} \Delta^{n-t}$.  Therefore, we may assume that there are at most $c_1$ vertices in disjoint cycles.  (All constants in the remainder of this proof will depend on $H$ but will be independent of $n$.)

After coloring the cycles, we look at the paths that are added iteratively.  If any path has length longer than some constant $l$, then by Lemma \ref{lem-long paths} we have $\hom(G,H) < \Delta^{n-2}$.  Since a path of length $k$, for $2 \leq k \leq l$, has at most $\Delta^k-1$ extensions to an $H$-coloring by Corollary \ref{cor-short paths}, if there are at least $c_2$ such paths then %there is a constant $c'_H$ so that
\[
\hom(G,H) < \left(\prod_{i=1}^{c_2} (\Delta^{k_i} -1) \right) \Delta^{n-\sum_{i=1}^{c_2} k_i} < \Delta^{n-2}.
\]
%and for large enough $c'_H$ this is smaller than $\Delta^{n-2}$.

So, we may assume that the decomposition of $G$ from Corollary \ref{cor-delta2decomposition} has fewer than $c_3$ vertices in either disjoint cycles or paths on $k$ vertices, with $2 \leq k \leq l$, and no paths on $k$ vertices, with $k > l$.  Therefore, the decomposition has at least $n-c_3$ vertices in paths on $1$ vertex.  Furthermore, each path on $1$ vertex must be attached to two of the at most $c_3$ vertices composing the disjoint cycles and the paths on at least $2$ vertices that are added.  By the pigeonhole principle there exists a $c_4 > 0$ and two vertices in $G$ with at least $c_4 n$ paths on $1$ vertex joining them.

\medskip 

We have shown that every edge-min-critical graph $G$ which does not have two vertices with at least $c_4 n$ paths on $1$ vertex joining them has $\hom(G,H) < \Delta^{n-2}$, and by Lemma \ref{lem-long paths} we have the same bound on $\hom(G,H)$ if $G$ has a path on $k$ vertices when $k > l$.  We now deal with the remaining edge-min-critical graphs $G$.

Let $w_1$ and $w_2$ denote the vertices in $G$ joined by at least $c_4 n$ paths on $1$ vertex.  Recall from Section \ref{sec-deltageneral} that $S(2,H)$ is the set of vectors in $V(H)^2$  with the property that the elements of the vector have $\Delta$ common neighbors, and $s(2,H) = |S(2,H)|$.  Suppose first that the colors on $w_1$ and $w_2$ are an element of $S(2,H)$.  If $G$ is different from $K_{2,n-2}$, then $w_1$, $w_2$, and the at least $c_4 n$ paths on $1$ vertex between them do not form all of $G$.  But then $G$ must contain either a cycle which does not include $w_1$ or $w_2$, or a path on $k$ vertices (for $2 \leq k \leq l_H$) from $w_i$ to $w_j$ for some $i,j \in \{1,2\}$.  By first coloring $w_1$ and $w_2$, then any remaining disjoint cycles, and finally the remaining vertices, Corollary \ref{cor-short paths} and (\ref{eqn-bound on cycles}) imply that there exists a $c_5 < 1$ such that there are at most
\begin{equation}\label{eqn-delta2 bound4}
s(2,H) c_5 \Delta^{n-2}
\end{equation}
$H$-colorings of $G$ of this type.

Now suppose that the colors on $w_1$ and $w_2$ are not an element of $S(2,H)$.  By first coloring $w_1$ and $w_2$, then any remaining disjoint cycles, and finally the remaining vertices, we have at most
\begin{equation}\label{eqn-delta2 bound1}
|V(H)|^2 \Delta^{n-c_4 n-2} (\Delta-1)^{c_4 n}
\end{equation}
$H$-colorings of $G$ of this type.

Combining (\ref{eqn-delta2 bound4}) and (\ref{eqn-delta2 bound1}) gives
\[
\hom(G,H) \leq |V(H)|^2 \Delta^{n-2} \left(\frac{\Delta-1}{\Delta}\right)^{c_4 n} + s(2,H) c_5 \Delta^{n-2} < s(2,H) \Delta^{n-2},
\]
with the last inequality holding for large enough $n$.  

We have shown that the only edge-min-critical graph which achieves equality is $K_{2,n-2}$. Arguing as in Case 1 of Section \ref{sec-deltageneral} shows that adding any edges to $K_{2,n-2}$ produces a graph $G$ with $\hom(G,H) < \hom(K_{2,n-2},H)$, which completes the proof. %We repeat an argument given from Case 1 in Section \ref{sec-deltageneral} to prove that this is the only graph which achieves equality.  Since $H$ cannot contain $K_{\Delta}^{\text{loop}}$ (by (\ref{eqn-delta2HCondition2})), there are two (possibly non-distinct) non-adjacent neighbors of a vertex $v \in V(H)$ with degree $\Delta$.  If any edge is added to $K_{2,n-2}$ (necessarily within a partition class), then it is impossible for any $H$-coloring to color the endpoints of that edge with the non-adjacent vertices, but such a coloring is possible in $K_{2,n-2}$.  %A similar argument shows that no edge can be added to the size $n-2$ partition class.  
%Therefore adding any edges to $K_{2,n-2}$ produces a graph $G$ with $\hom(G,H) < \hom(K_{2,n-2},H)$.

%============================================================================
%Section Concluding remarks
%============================================================================

\section{Concluding Remarks}\label{sec-concludingremarks}

We now briefly describe how to generalize the results in this paper to weighted $H$-colorings.  Suppose that $\Lambda = \{\lambda_v : v \in V(H) \}$ is a set of positive weights indexed by $V(H)$.  Let 
\[
d_\Lambda(v) = \sum_{w \sim_H v} \lambda_w
\]
(noting that a loop on $v$ will include one $\lambda_v$ term in the sum) and 
\[
\Delta_\Lambda = \max_{v \in V(H)} d_\Lambda(v).
\]
Furthermore, for $i,j \in V(H)$, let $A_\Lambda$ be the matrix with $ij$ entry $\sqrt{\lambda_i \lambda_j}\textbf{1}_{i \sim_H j}$.  Then the results of Theorems \ref{thm-2 regular graphs}, \ref{thm-Strong delta1}, \ref{thm-Strong delta2}, and \ref{thm-ExtremalHColorings} hold for weighted $H$-colorings by changing $d(v)$ to $d_\Lambda(v)$, $\Delta$ to $\Delta_\Lambda$, $A$ to $A_\Lambda$, and $\hom(G,H)$ to $Z_\Lambda(G,H)$.  For example, the condition $\sum_{v \in V(H)} d(v) < \Delta^2$ is replaced by $\sum_{v \in V(H)} \lambda_v d_\Lambda(v) < (\Delta_\Lambda)^2$.  The generalization is straightforward and we leave the details to the interested reader.

\medskip

The proofs of Theorems \ref{thm-Strong delta1} and \ref{thm-Strong delta2} depend heavily on analyzing the structure of edge-min-critical graphs.  For $\delta\geq 3$, there is no helpful structural characterization of these graphs.  A nice answer to the following question might help answer Question \ref{conj-H Colorings Min Degree} for other values of $\delta$. 

\begin{question}
What can be said about the structure of edge-min-critical graphs in $\calG(n,\delta)$ for $\delta \geq 3$?
\end{question}

%Based on the evidence given for Conjecture \ref{conj-H Colorings Min Degree}, we have the following (weaker) conjecture.
%\begin{conjecture}\label{conj-necessary and sufficient H}
%Fix $\delta$, $H$, and $n > c(\delta,H)$.  Then, for any $G \in \calG(n,\delta)$, 
%\[
%\hom(G,H) \leq \hom(K_{\delta,n-\delta},H)  \iff \max \{ \hom(K_{\delta,\delta},H)^{\frac{1}{2\delta}}, \hom(K_{\delta+1},H)^{\frac{1}{2\delta}} \} < \Delta.
%\]
%\end{conjecture}
In light of Theorem \ref{thm-max degree condition}, it would be interesting to consider a maximum degree condition $D$ in addition to a minimal degree condition $\delta$ when $\delta > 1$.  Again writing $\calG(n,\delta,D)$ for the set of graphs with minimum degree $\delta$ and maximum degree at most $D$, if $D < n-\delta$ then $K_{\delta,n-\delta} \notin \calG(n,\delta,D)$.  It is possible that, as in the $\delta=1$ case, the graph $K_{\delta,n-\delta}$ is replaced by $\frac{n}{\delta+D} K_{\delta,D}$ in $\calG(n,\delta,D).$ % (see e.g. \cite{AlexanderCutlerMink,GalvinSmallDegree,Kahn1}).  Let $\calG(n,\delta,\Delta)$ denote the set of graphs on $n$ vertices with minimum degree $\delta$ and maximum degree at most $\Delta$ (we emphasize that here, and for the remainder of this section, $\Delta$ refers to the maximal degree of graphs $G$).  Which graphs $G \in \calG(n,\delta,\Delta)$ maximize $\hom(G,H)$?  This question is especially interesting for the graphs $H$ with the property that $\hom(G,H) \leq \hom(K_{\delta,n-\delta},H)$ for all $G \in \calG(n,\delta)$, as $K_{\delta,n-\delta} \in \calG(n,\delta,\Delta)$ only when $\Delta \geq n-\delta$. %Since $K_{\delta,n-\delta}$ has maximum degree $n-\delta$, it would be interesting to ask which $n$-vertex graph(s) with a fixed minimum degree and maximum degree maximize $\hom(G,H)$.  
%Notice that we can still delete any edge from a graph $G \in \calG(n,\delta,\Delta)$ and remain in $\calG(n,\delta,\Delta)$ as long the edge deletion does not lower the minimum degree.  Therefore, for $\delta=1$, the proof of Theorem \ref{thm-Strong delta1} shows that for any $G \in \calG(n,1,\Delta)$,
%suppose that we start with an $n$-vertex graph $G$ with minimum degree $1$ and maximum degree at most $\Delta$.  If we add a disjoint copy of $K_{1,\Delta}$ to $G$, we can then delete edges from $G$ while preserving (careful --- can't do edge critical.  Want to add a disjoint component that is $K_{1,\Delta}$, then edge-critical, then remove the component...)
%\[
%\hom(G,H) \leq \max\{ \hom(K_2,H)^{\frac{n}{2}}, \hom(K_{1,\Delta})^{\frac{n}{1+\Delta}} \}.
%\]
This leads to the following natural question. 
%the following question is interesting for those $H$ for which $\hom(G,H)$ is conjectured to be maximized by $G=K_{\delta,n-\delta}$.

\begin{question}
For which fixed $\delta$, $D$, $H$, and $n$ is it true that for any $G \in \calG(n,\delta,D)$, %  Is it true that for any $G \in \calG(n,\delta,D)$
\[
\hom(G,H) \leq \max \{ \hom(K_{\delta+1},H)^{\frac{n}{\delta+1}}, \hom(K_{\delta,\delta},H)^{\frac{n}{2\delta}}, \hom(K_{\delta,D},H)^{\frac{n}{\delta + D}} \}?
\]
%which $n$-vertex graph(s) with a fixed minimum and maximum degree maximizes $\hom(G,H)$? %Remark that this is easy for $\delta=1$
\end{question}

%We remark that the proof of Theorem \ref{thm-Strong delta1} can be used to show that

%A natural direction to 

%\begin{question}
%Reference conjecture from Alexander/Cutler/Mink paper --- what replaces $K_{\delta,n-\delta}$ in the conjecture if we consider $\calG(n,\delta,\Delta_G)$, the family of $n$-vertex graphs with minimum degree $\delta$ and maximum degree $\Delta_G$?  Is it $\frac{n}{\delta + \Delta_G} K_{\delta,\Delta_G}$?
%\end{question}
%

\medskip

\noindent \textbf{Acknowledgement:} The author is grateful to Igor Rivin for suggesting the use of Proposition \ref{thm-Lp}, to David Galvin for a number of helpful discussions, and to the anonymous referees for several helpful suggestions.

%=================================================================================
%References
%=================================================================================


\begin{thebibliography}{99}

\bibitem{AlexanderCutlerMink}
J. Alexander, J. Cutler, and T. Mink, Independent sets in graphs with given minimum degree, {\em Electron. J. Combin.} {\bf 19(3)} (2012), \#P37.

\bibitem{AlexanderMink}
J. Alexander and T. Mink, A new method for enumerating independent sets of a fixed size in general graphs, arXiv:1308.3242.

\bibitem{CutlerRadcliffe2}
J. Cutler and A.J. Radcliffe, Extremal graphs for homomorphisms, {\em J. Graph Theory} {\bf 67} (2011), 261-284.

\bibitem{CutlerRadcliffe3}
J. Cutler and A.J. Radcliffe, Extremal graphs for homomorphisms II, {\em J. Graph Theory} {\bf 76} (2014), 42-59.

\bibitem{CutlerRadcliffe4}
J. Cutler and A.J. Radcliffe, The maximum number of complete graphs in a graph with given maximum degree, {\em J. Combin. Theory Ser. B} {\bf 104} (2014), 60-71.

\bibitem{EngbersGalvin1}
J. Engbers and D. Galvin, $H$-coloring bipartite graphs, {\em J. Combin. Theory Ser. B} {\bf 102} (2012), 726-742.

\bibitem{EngbersGalvin3}
J. Engbers and D. Galvin, Counting independent sets of a fixed size in graphs with given minimum degree, {\em J. Graph Theory} {\bf 76} (2014), 149-168.

\bibitem{GalvinTwoProblems}
D. Galvin, Two problems on independent sets in graphs, {\em Discrete Math.} {\bf 311} (2011), 2105-2112.

\bibitem{GalvinHColoringRegularGraphs}
D. Galvin, Maximizing $H$-colorings of regular graphs, {\em J. Graph Theory} {\bf 73} (2013), 66-84.

\bibitem{GalvinProperColoringRegularGraphs}
D. Galvin, Counting colorings of a regular graph, to appear in {\em Graphs Combin.}, DOI 10.1007/s00373-013-1403-z.

\bibitem{GalvinSernau}
D. Galvin, personal communication.

\bibitem{GalvinMartinelliRamananTetali}
D. Galvin, F. Martinelli, K. Ramanan, and P. Tetali, The multi-state hard core model on a regular tree, {\em SIAM J. Discrete Math} {\bf 25} (2011), 894-916.

\bibitem{GalvinTetali}
D. Galvin and P. Tetali, On weighted graph homomorphisms, {\em Graphs, Morphisms, and Statistical Physics, DIMACS Ser. in Discrete Math. Theoret. Comput. Sci.} {\bf 63} (2004), 97-104.

\bibitem{GalvinSmallDegree}
D. Galvin and Y. Zhao, The number of independent sets in a graph with small maximum degree, {\em Graphs Combin.} {\bf 27} (2011), 177-186.

\bibitem{GanLohSudakov}
W. Gan, P.-S. Loh, and B. Sudakov, Maximizing the number of independent sets of a fixed size, arXiv:1311.4147.

\bibitem{Kahn1}
J. Kahn, An entropy approach to the hard-core model on bipartite graphs, {\em Combin. Probab. and Comput.} {\bf 10} (2001), 219-237.

\bibitem{Kahn3}
J. Kahn, Entropy, independent sets and antichains: a new approach to Dedekind's problem, {\em Proc. Amer. Math. Soc.} {\bf 130(2)} (2002), 371-378.

\bibitem{LawMcDiarmid}
H.-F. Law and C. McDiarmid, On independent sets in graphs with given minimum degree, {\em Combin. Probab. and Comput.} {\bf 22} (2013), 874-884.

\bibitem{Lazebnik}
F. Lazebnik, On the greatest number of $2$ and $3$ colorings of a $(v,e)$-graph, {\em J. Graph Theory} {\bf 13(2)} (1989), 203-214.

\bibitem{Linial}
N. Linial, Legal coloring of graphs, {\em Combinatorica} {\bf 6(1)} (1986), 49-54.

\bibitem{LohPikhurkoSudakov}
P.-S. Loh, O. Pikhurko, and B. Sudakov, Maximizing the number of $q$-colorings, {\em Proc. Lon. Math. Soc.} {\bf 101} (2010), 655-696.

\bibitem{MitraRamananSengupta}
K. Ramanan, A. Sengupta, I. Ziedins, and P. Mitra, Markov random field models of multicasting in tree networks, {\em Adv. in Appl. Probab.} {\bf 34} (2002), 58-84.

\bibitem{Tao}
T. Tao, An epsilon of room, I: real analysis: pages from year three of a mathematical blog, Graduate studies in mathematics, vol 117, AMS, Providence, RI, 2009.

\bibitem{Wilf}
H. Wilf, Backtrack: an $O(1)$ expected time algorithm for the graph coloring problem, {\em Inform. Process. Lett.} {\bf 18(3)} (1984), 119-121.

\bibitem{Zhao}
Y. Zhao, The number of independent sets in a regular graph, {\em Combin. Probab. and Comput.} {\bf 19} (2010), 315-320.

\bibitem{Zhao2}
Y. Zhao, The bipartite swapping trick on graph homomorphisms, {\em SIAM J. Discrete Math} {\bf 25} (2011), 660-680.

\end{thebibliography}
\end{document}